\documentclass[11pt]{amsart}

\usepackage{epigamath}

\usepackage[english]{babel}

\numberwithin{equation}{section}

\usepackage[utf8]{inputenc}
\usepackage[T1]{fontenc}
\usepackage{amssymb,amsmath,amsfonts}
\usepackage[english]{babel}
\usepackage{mathrsfs}
\usepackage{tikz-cd}
\usepackage{verbatim}
\usepackage{color}
\usepackage{mathtools}

\usepackage{tikz}

\usepackage[shortlabels]{enumitem}
\setlist[enumerate]{label=\rm{(\arabic*)}, ref={\rm\arabic*}}
\setlist[enumerate,2]{label=\rm{(\roman*)}, ref={\rm\roman*}}

\usepackage[all]{xy}

\theoremstyle{plain}
\newtheorem{theorem}{Theorem}[section]
\newtheorem{corollary}[theorem]{Corollary}
\newtheorem{proposition}[theorem]{Proposition}
\newtheorem{lemma}[theorem]{Lemma}

\newtheorem{theoremA}{Theorem}

\newtheorem{corollaryA}[theoremA]{Corollary}

\theoremstyle{definition}
\newtheorem{definition}[theorem]{Definition}
\newtheorem*{conventions}{Conventions}

\theoremstyle{remark}
\newtheorem{remark}[theorem]{Remark}
\newtheorem{example}[theorem]{Example}

\def\Autzero{\operatorname{Aut}^\circ}
\def\Aut{\operatorname{Aut}}

\def\PP{\mathbb{P}}
\def\FF{\mathbb{F}}
\def\Pic{\operatorname{Pic}}
\def\kk{\mathbf{k}}

\def\seg{\mathfrak{S}}
\def\PGL{\operatorname{PGL}}
\def\Bir{\operatorname{Bir}}
\def\Fix{\operatorname{Fix}}
\def\inv{\operatorname{inv}}

\def\lra{\longrightarrow}

\newcommand{\var}[1]{{\operatorname{#1}}}
\newcommand{\xdashrightarrow}{\joinrel{\var{-}}\joinrel\dashrightarrow}

\EpigaVolumeYear{7}{2023} \EpigaArticleNr{13} \ReceivedOn{November 23, 2021}
\InFinalFormOn{September 9, 2022}
\AcceptedOn{January 14, 2023}

\title{Algebraic subgroups of the group of birational transformations of ruled surfaces}
\titlemark{Algebraic subgroups of $\boldsymbol{\operatorname{Bir}(C\times \PP^1)}$}

\author{Pascal Fong}
\address{Universit\"at Basel, Departement Mathematik und Informatik, Spiegelgasse 1, CH--4051 Basel, Switzerland}
\email{pascal.fong@unibas.ch}

\authormark{P.~Fong}

\AbstractInEnglish{We classify the maximal algebraic subgroups of $\operatorname{Bir}(C\times \PP^1)$, when $C$ is a smooth projective curve of positive genus.} 

\acknowledgement{The author acknowledges support by the Swiss National Science Foundation Grant “Geometrically ruled surfaces” 200020–192217.}

\begin{document}


\maketitle

\begin{prelims}

\DisplayAbstractInEnglish

\bigskip

\DisplayKeyWords

\medskip

\DisplayMSCclass

\end{prelims}


\newpage

\setcounter{tocdepth}{1}

\tableofcontents


\section{Introduction}

In this article, all varieties are defined over an algebraically closed field $\kk$, algebraic groups are smooth group schemes of finite type (or, equivalently, reduced group schemes of finite type), and $C$ denotes a smooth projective curve of genus $g$. The main results, namely Theorem~\ref{A} and Corollary~\ref{D}, hold when the characteristic of $\kk$ is different from two. When $X$ is a projective variety, the automorphism group of $X$ is the group of $\kk$-rational points of a group scheme (see \cite{Matsumura_Oort}), and we only consider its reduced structure.

The study of algebraic subgroups of the group of birational transformations started with \cite{Enriques}, where the author classified the maximal connected algebraic subgroups of $\operatorname{Bir}(\PP^2)$. More recently, the maximal algebraic subgroups of $\operatorname{Bir}(\PP^2)$ have been classified; see \cite{Blanc}. The purpose of this text is to study the algebraic subgroups of $\operatorname{Bir}(C\times \PP^1)$ when $g\geq 1$, which will complete the classification for surfaces of Kodaira dimension $-\infty$. 

Let $G$ be an algebraic subgroup of $\operatorname{Bir}(C\times \PP^1)$. The strategy is classical: first regularize the action of~$G$, find a $G$-equivariant completion, and run a $G$-equivariant minimal model program (MMP) to embed $G$ in the automorphism group of a $G$-minimal fibration. The equivariant completion from Sumihiro \cite{Sumihiro1,Sumihiro2} works for linear algebraic groups; therefore, his results cannot be applied in our setting. Recently, Brion proved the existence of an equivariant completion for connected (not necessarily linear) algebraic groups acting birationally on integral varieties; see \cite[Corollary 3]{Brion_alggpactions}. Using his results, we find an equivariant completion for (not necessarily linear or connected) algebraic groups acting on surfaces (see Proposition~\ref{regularization}). Then we re-prove the $G$-equivariant MMP (Proposition~\ref{Iskovskikh}), which is a folklore result (see \textit{e.g.}\ \cite[Example 2.18]{KollarMori}), by using only elementary arguments. We are left with studying the automorphism groups of conic bundles. Following the ideas of \cite{Blanc}, we prove Propositions~\ref{reduction} and~\ref{key}, which reduce the study to the cases of ruled surfaces, exceptional conic bundles and $(\mathbb{Z}/2\mathbb{Z})^2$-conic bundles (see Section~\ref{definitionsconic} for definitions). The Segre invariant $\seg(X)$ of a ruled surface $X$ (see Definition~\ref{defSegre}) was introduced in \cite{Maruyama2} and \cite{Maruyama} for the classification of ruled surfaces and their automorphisms. The ruled surfaces $\mathcal{A}_0$ and $\mathcal{A}_1$ in Theorem~\ref{A}\eqref{A-4}, \eqref{A-5} are the only indecomposable $\PP^1$-bundles over $C$ up to $C$-isomorphism, when $C$ is an elliptic curve (see Definition~\ref{defdecomposable} and \cite[Theorem 11]{Atiyah} or \cite[Theorem V.2.15]{Hartshorne}). Combining techniques from Blanc and results of Maruyama, we prove the following theorem. 

\begin{theoremA}\label{A}
	Let $\kk$ be an algebraically closed field of characteristic different from two, and let $C$ be a smooth projective curve of genus $g\geq 1$. The following algebraic subgroups of $\operatorname{Bir}(C\times \PP^1)$ are maximal:
	\begin{enumerate}
		\item\label{A-1} $\Aut(C\times \PP^1)\simeq \Aut(C)\times  \operatorname{PGL}(2,\kk)$; 
		\item\label{A-2} $\Aut(X)$, where $X$ is an exceptional conic bundle over $C$ which is the blowup of a decomposable ruled surface $\pi\colon \PP(\mathcal{O}_C(D)\oplus \mathcal{O}_C)\to C$ along $F=\{p_1,p_2,\ldots,p_{2\deg(D)}\}$ lying in two disjoint sections $s_1$ and $s_2$ of $\pi$, and such that $-2D$ is linearly equivalent to
		$$
		\sum_{p\in s_1\cap F} \pi(p) - \sum_{p\in s_2\cap F} \pi(p);
		$$
		$($Then $\Aut(X)$ fits into an exact sequence
		$$ 1 \lra \mathbb{G}_m\rtimes \mathbb{Z}/2\mathbb{Z} \lra \Aut(X) \lra H,$$
		where $H$ is the finite subgroup of $\Aut(C)$ preserving the image of the singular fibres.$)$

	      \item\label{A-3} $\Aut(X)$, where $X$ is a $(\mathbb{Z}/2\mathbb{Z})^2$-conic bundle with at least one singular fibre;

                \noindent$($Then $\Aut(X)$ fits into an exact sequence
		$$ 1 \lra (\mathbb{Z}/2\mathbb{Z})^2 \lra \Aut(X) \lra H,$$
		where $H$ is the finite subgroup of $\Aut(C)$ preserving the image of the singular fibres.$)$
	     
 \item\label{A-4} $\Aut(X)$, where $X$ is a $(\mathbb{Z}/2\mathbb{Z})^2$-ruled surface $($consequently, $\seg(X)>0)$;

                \noindent $($Then $\Aut(X)$ fits into an exact sequence
		$$ 1 \lra (\mathbb{Z}/2\mathbb{Z})^2 \lra \Aut(X) \lra \Aut(C).$$
		Moreover, if $g=1$, there exists a unique $(\mathbb{Z}/2\mathbb{Z})^2$-ruled surface over $C$ denoted by $\mathcal{A}_1$ which satisfies $\seg(\mathcal{A}_1)=1$ and for which $\Aut(\mathcal{A}_1)$ fits into an exact sequence
		$$ 1 \lra (\mathbb{Z}/2\mathbb{Z})^2 \lra \Aut(\mathcal{A}_1) \lra \Aut(C) \to 1.)$$

	      \item\label{A-5} $\Aut(\mathcal{A}_0)$, where $\mathcal{A}_0$ is the unique indecomposable ruled surface over $C$ with Segre invariant $0$ when $g=1$;

                \noindent $($Then there exists an exact sequence
		$$ 1 \lra \mathbb{G}_a \lra \Aut(\mathcal{A}_0) \lra \Aut(C)\lra 1.)$$
	      \item\label{A-6}  $\Aut(X)$, where $X\simeq \PP(\mathcal{O}_C(D)\oplus \mathcal{O}_C)$ is a non-trivial decomposable ruled surface over $C$ with $\deg(D)=0$ $($or, equivalently, $\seg(X)=0)$, with the additional assumption that $2D$ is principal if $g\geq 2$.

                \noindent$($Then $\Aut(X)$ fits into an exact sequence
		$$ 1 \lra G \lra \Aut(X) \lra \Aut(C),$$
		where $G=\mathbb{G}_m\rtimes \mathbb{Z}/2\mathbb{Z}$ if\, $2D$ is principal, else $G=\mathbb{G}_m$.$)$
	\end{enumerate}
	Moreover, any maximal algebraic subgroup of\, $\operatorname{Bir}(C\times \PP^1)$ is conjugate to one in the list above.
\end{theoremA}

By Corollary~\ref{exceptionalnotmax}, there exist exceptional conic bundles $X\to C$, where $C$ is a curve of positive genus, such that $\Aut(X)$ is not a maximal algebraic subgroup of $\Bir(C\times \PP^1)$. This does not happen when the base curve is rational: the automorphism group of an exceptional conic bundle over $\PP^1$ is always maximal (if the number of singular fibres is at least four, see \cite[Theorem 1(2)]{Blanc}; else the number of singular fibres equals two, and the result follows from \cite[Theorem 2(3)]{Blanc}). Moreover, the cases \eqref{A-4}, \eqref{A-5} and \eqref{A-6} of Theorem~\ref{A} do not exist when the base curve is rational: the Segre invariant of a ruled surface $\pi\colon S \to \PP^1$ is always non-positive (see \cite{HM} and \cite[Proposition 2.18(1)]{fong}) and equals $0$ if and only if $\pi$ is trivial.

From the classification of Blanc, it follows that every algebraic subgroup of $\operatorname{Bir}(\PP^2)$ is conjugate to a subgroup of a maximal one. This does not hold anymore for algebraic subgroups of $\operatorname{Bir}(C\times \PP^1)$ when $C$ has positive genus. The following corollary is an analogue of \cite[Theorem C]{fong} for surfaces of Kodaira dimension $-\infty$.

\begin{corollaryA}\label{D}
	Let $\kk$ be an algebraically closed field of characteristic different from two, and let $X$ be a surface of Kodaira dimension $-\infty$. Then, every algebraic subgroup of $\operatorname{Bir}(X)$ is contained in a maximal one if and only if $X$ is rational.
\end{corollaryA}

\begin{conventions}
	Unless stated otherwise, all varieties are smooth and projective, and $C$ is a smooth projective curve.
\end{conventions}

\subsection*{Acknowledgments}
	The author is thankful to J\'er\'emy Blanc, Michel Brion, Ronan Terpereau and Sokratis Zikas for helpful discussions and comments. The author is also grateful to the anonymous referees for their careful reading and interesting remarks.

\section{Preliminaries}

\subsection{Regularization and relative minimal fibrations}

\begin{definition}\label{definitionsconic}
	Let $C$ be a curve.
	\begin{enumerate}
		\item A \emph{ruled surface} over $C$ is a morphism $\pi\colon S\to C$ such that each fibre is isomorphic to $\PP^1$. 
		\item A \emph{conic bundle} over $C$ is a morphism $\kappa\colon X\to C$ such that all fibres are isomorphic to $\PP^1$, except finitely many (possibly $0$) which are called \emph{singular fibres} and are transverse unions of two $(-1)$-curves.
		\item A conic bundle $\kappa\colon X\to C$ is an \emph{exceptional conic bundle} over $C$ if there exists an $n\geq 1$ such that $\kappa$ has exactly $2n$ singular fibres and two sections of self-intersection $-n$. 
		\item If $\kappa\colon X \to C$ is a conic bundle, we denote by $\operatorname{Bir}_C(X)$ the subgroup of $\operatorname{Bir}(X)$ which consists of the elements $f\in \operatorname{Bir}(X)$ such that $\kappa f = \kappa$. We also define $\Aut_C(X)=\Aut(X) \cap \operatorname{Bir}_C(X)$.
		\item A conic bundle $\kappa\colon X\to C$ is a $(\mathbb{Z}/2\mathbb{Z})^2$-\emph{conic bundle} over $C$ if $\Aut_C(X)\simeq (\mathbb{Z}/2\mathbb{Z})^2$ and each non-trivial involution in this group fixes pointwise an irreducible curve, which is a $2$-to-$1$ cover of $C$ ramified above an even positive number of points.
		\item A $(\mathbb{Z}/2\mathbb{Z})^2$-\emph{ruled surface} over $C$ is a ruled surface over $C$ which is also a $(\mathbb{Z}/2\mathbb{Z})^2$-conic bundle over~$C$. 
	\end{enumerate}
\end{definition}

\begin{remark}
	Assume that $C$ has positive genus, and let $\pi \colon X\to C$ be a conic bundle. Let $f$ be a smooth fibre of $\pi$ and $\alpha\in \Aut(X)$. Since $(\pi \alpha)_{|f}\colon f\simeq \PP^1 \to C$ is constant, it follows that $\alpha(f)$ is also a smooth fibre of $\pi$. The set of singular fibres is preserved by $\Aut(X)$, and $\pi$ induces a morphism of group schemes $\pi_*\colon \Aut(X)\to \Aut(C)$. This implies that every automorphism of $X$ preserves the conic bundle structure.
\end{remark}

\begin{definition}\label{defequivariantminimal}
	Let $X$ be a surface and $G$ be an algebraic subgroup of $\Aut(X)$. 
	\begin{enumerate}
		\item A birational map $\phi\colon X\dashrightarrow Y$ is $G$-\emph{equivariant} if $\phi G\phi^{-1}\subset \Aut(Y)$.
		\item The pair $(G,X)$ is \emph{minimal} if every $G$-equivariant birational morphism $X\to X'$, where $X'$ is a surface, is an isomorphism.
	\end{enumerate}
\end{definition}

The classical approach to study algebraic subgroups uses the regularization theorem of Weil \cite{Weil} (see also \cite{Zaitsev} or \cite{Kraft} for modern proofs). By \cite[Theorem 1]{Brion_models}, the regularization of $X$ contains a $G$-stable dense open subset $U$ which is smooth and quasi-projective. Then by \cite[Theorem 2]{Brion_models}, $U$ admits a $G$-equivariant completion by a normal projective $G$-variety, that we can assume smooth by a $G$-equivariant desingularization (see \cite{Lipman}). 

In the following lemma and proposition, we give an elementary proof of the existence of an equivariant completion for surfaces equipped with the action of an algebraic group $G$, not necessarily connected or linear, without using results of \cite{Brion_models}. 

\begin{lemma}\label{resolutionindeterminacies}
	Let $X$ be a surface and $G$ be an algebraic subgroup of\, $\operatorname{Bir}(X)$ such that $G^\circ$ acts regularly on $X$. Denote by $\operatorname{Bs}(X)$ the set of base points of the $G$-action, including the infinitely near ones. Then $\operatorname{Bs}(X)$ is finite, and the action of\, $G$ lifts to a regular action on the blowup of\, $X$ at $\operatorname{Bs}(X)$.
\end{lemma}

\begin{proof}
	The set $G/G^\circ$ is finite; we can write $G=G_0 \sqcup G_1 \sqcup \cdots \sqcup G_n$, where $G^\circ = G_0$ and each $G_i$ is a connected component of $G$. For each $i$, we fix $g_i\in G_i$ and get that $g_iG^\circ = G_i$. Let $\operatorname{Bs}(g_i)$ be the set of base points of $g_i$, including the infinitely near ones; this set is finite because $X$ is smooth and projective. The subgroup $G^\circ \subset G$ is normal; it follows that every element $g_i'\in G_i$ equals $gg_i$ for some $g\in G^\circ$. Since $G^\circ$ acts regularly, $\operatorname{Bs}(G_i)= \operatorname{Bs}(g_i)$ for each $i$, and this implies that $\operatorname{Bs}(X)= \bigcup_{i=1,\ldots, n} \operatorname{Bs}(g_i)$ is also finite. Moreover, for each $g\in G^\circ$, there exists a $\tilde{g}\in G^\circ$ such that $g_ig = \tilde{g}g_i$. Then $g^{-1}(\operatorname{Bs}(g_i))=\operatorname{Bs}(g_i)$, so $G^\circ$ preserves $\operatorname{Bs}(X)$. Since $G^\circ$ is connected and $\operatorname{Bs}(X)$ is finite, this implies that $G^\circ$ acts trivially on $\operatorname{Bs}(X)$.
	
	 If $\operatorname{Bs}(X)$ is empty, the result holds. Suppose that $\operatorname{Bs}(X)\neq \emptyset$. Let $p\in \operatorname{Bs}(X)\cap X$ be a proper base point and $\eta\colon X_p \to  X$ be the blowup of $X$ at $p$. We consider the action of $G$ on $X_p$ obtained by conjugation. As $p$ is fixed by $G^\circ$, the algebraic group $G^\circ$ still acts regularly on $X_p$. We prove that each element $q\in\operatorname{Bs}(X_p)$ corresponds via $\eta$ to an element of $\operatorname{Bs}(X)$. Let $q\in \operatorname{Bs}(X_p)$; there exist a surface $Y$ such that $q\in Y$ and a  birational morphism $\pi \colon Y\to X_p$ such that $q$ is a base point of $\eta^{-1} g_i \eta \pi$. Let $W$ be a smooth projective surface with birational morphisms $\alpha\colon W\to Y$ and $\beta \colon W\to X_p$ such that $\beta \alpha^{-1}$ is a minimal resolution of $\eta^{-1} g_i \eta \pi$. The following diagram is commutative:
	\[
	\begin{tikzcd} [column sep=3em,row sep = 2em,/tikz/row 1/.style={row sep=0.1em},/tikz/column 1/.style={column sep=0em}]
		&& W \arrow[ld, "\alpha"] \arrow[rdd,"\beta",swap]& \\
		q\in& Y\arrow[rd,"\pi"]\arrow[rdd,"\eta\pi",swap] & & \\
		&& X_p \arrow[d,"\eta"]\arrow[r,dashed,"\eta^{-1}g_i\eta",swap]& X_p \arrow[d,"\eta"] \\
		&& X \arrow[r,dashed,"g_i",swap]& X\rlap{.}
	\end{tikzcd}
	\] 
	Since $q$ is a base point of $\eta^{-1} g_i \eta \pi$, it must be blown up by $\alpha$. There exists a $(-1)$-curve $\widetilde{C}$ in $W$ contracted by $\alpha$ to $q$ and such that its image by $\beta$ is a curve $C$ in $X_p$. If the image of $C$ by $\eta$ is a curve in $X$, then $q$ is a base point of $g_i$. Else $C$ is contracted by $\eta$; \textit{i.e.}\ $C$ is the exceptional divisor of $\eta$. As $C^2=\widetilde{C}^2=-1$, the morphism $\beta$ is an isomorphism $\widetilde{U}\to U$, where $\widetilde{U}\subset W$ and $U\subset X_p$ are open subsets containing $\widetilde{C}$ and $C$. Let $j$ be such that $p$ is a proper base point of $g_j$. Let $W'$ be a smooth projective surface with birational morphisms $\alpha' \colon W'\to X_p$ and $\beta'\colon W'\to X$ such that $\beta' \alpha'^{-1}\eta^{-1}$ is a minimal resolution of $g_j$. We obtain the following commutative diagram:
	\[
	\begin{tikzcd} [column sep=3em,row sep = 2em,/tikz/row 1/.style={row sep=0.1em},/tikz/column 1/.style={column sep=0em}]
		&& W \arrow[ld, "\alpha"] \arrow[rdd,"\beta",swap] & W' \arrow[dd,"\alpha' "] \arrow[rddd,"\beta'"]\\
		q\in& Y\arrow[rd,"\pi"]\arrow[rdd,"\eta\pi",swap] & & \\
		&& X_p \arrow[d,"\eta"]\arrow[r,dashed,"\eta^{-1}g_i\eta",swap]& X_p \arrow[d,"\eta"] \\
		&& X \arrow[r,dashed,"g_i",swap]& X \arrow[r,dashed,"g_j",swap] & X\rlap{.}
	\end{tikzcd}
	\] 
	There exists a $(-1)$-curve $C'$ in $W'$ contracted to $p$ by $\eta\alpha'$ and such that its image by $\beta'$ is a curve in~$X$. The image of $C'$ by $\alpha'$ is $C$ or a point of $C$. Since $\beta\colon \widetilde{U}\to U$ is an isomorphism, this implies that $\beta^{-1}\alpha'\colon W' \dashrightarrow W$ is defined on a neighbourhood of $C'$ and sends $C'$ onto either $\tilde{C}$ or a point of $\tilde{C}$. Hence $\alpha\beta^{-1}\alpha'\colon W' \dashrightarrow Y$ contracts $C'$ onto $q$, so $q\in \operatorname{Bs}(g_jg_i)$.
	
	Let $\operatorname{Bs}(X_p)$ be the set of base points of $\eta^{-1} G \eta$, including the infinitely near ones. We have shown that if $q\in \operatorname{Bs}(X_p)$, then $q\in \operatorname{Bs}(X)$. The map $\operatorname{Bs}(X_p)\to \operatorname{Bs}(X)\setminus \{p\}$ sending the infinitesimal base point $(q,\pi)$ to $(q,\eta\pi)$ is injective. Conversely, if $q\in \operatorname{Bs}(g_i)$ and $q\neq p$, then $\eta^{-1}(q)\in \operatorname{Bs}(\eta^{-1}g_i\eta)$. Therefore, $\operatorname{Bs}(X_p)\simeq \operatorname{Bs}(X)\setminus \{p\}$. Proceeding by induction, the blowup of all elements of $\operatorname{Bs}$ gives rise to a surface on which $G$ acts regularly.
\end{proof}

\begin{proposition}\label{regularization}
	Let $X$ be a surface and $G$ be an algebraic subgroup of\, $\operatorname{Bir}(X)$. Then there exists a smooth projective surface $Y$ with a birational map $\psi\colon X\dashrightarrow Y$ such that $\psi G\psi^{-1}$ is an algebraic subgroup of $\Aut(Y)$.
\end{proposition}

\begin{proof}
	Apply \cite[Corollary 3]{Brion_alggpactions} to $X$ equipped with the action of the connected component of the identity $G^\circ\subset G$. There exists a normal projective surface $Z$ with a birational map $\phi\colon X\dashrightarrow Z$ such that $\phi G^\circ \phi^{-1} \subset \Autzero(Z)$. By an equivariant desingularization, we can also assume that $Z$ is smooth; see \cite{Lipman}. Let $H = \phi G \phi^{-1}$ and $\eta\colon Y\to Z$ be the blowup of $Z$ at $\operatorname{Bs}(Z)$. By Lemma~\ref{resolutionindeterminacies}, the action of $H$ lifts to a regular action on $Y$. Then $\eta^{-1} H \eta \subset \Aut(Y)$ is a closed subgroup which is an algebraic subgroup of $\operatorname{Bir}(Y)$. Take $\psi=\eta^{-1}\phi$; we get that $\psi G \psi^{-1}$ is an algebraic subgroup of $\Aut(Y)$. 
\end{proof}

The next result is also known, see \textit{e.g.}~\cite[Example 2.18]{KollarMori}; we re-prove it in our specific situation using elementary arguments.

\begin{proposition}\label{Iskovskikh}
	Let $C$ be a curve of positive genus, and let $X$ be a surface birationally equivalent to $C\times \PP^1$. Let $G$ be an algebraic subgroup of $\Aut(X)$. If $(G,X)$ is minimal $($see Definition~\ref{defequivariantminimal}\,$)$, then $X$ is a conic bundle over $C$.
\end{proposition}

\begin{proof}
	Since $X$ is birational to $C\times \PP^1$, there exist a morphism $\kappa\colon X\to C$ and a birational map $\phi\colon C\times \PP^1\dashrightarrow X$ such that $\kappa\phi = p_1$, where $p_1\colon C\times \PP^1\to C$ denotes the projection onto the first factor. In particular, $\phi$ is a finite composite of blowups and contractions, and there exists a non-empty open $U\subset C$ such that $\phi_{|U\times \PP^1}$ is an isomorphism. Let $p\in C\setminus U$; it remains to see that $\kappa^{-1}(p)$ either is isomorphic to $\PP^1$ or is the transverse union of two $(-1)$-curves. Since $X$ is the blowup of a ruled surface $S$ in finitely many  (possibly infinitely close) points, we can write $\kappa^{-1}(p) = E_1\cup \cdots \cup E_n$, where: 
	\begin{itemize}[wide]
		\item Each $E_i$ is isomorphic to $\PP^1$.
		\item For all distinct $i,j$, $E_i$ and $E_j$ either intersect transversely at a point or are disjoint.
	\end{itemize}
	If $n= 1$, then $\kappa^{-1}(p)$ is a smooth fibre isomorphic to $\PP^1$. If $n=2$, then $E_1$ and $E_2$ intersect transversely in one point. Because there is a contraction to the ruled surface $S$, either $E_1$ or $E_2$ can be contracted. Therefore, $E_1^2=E_2^2=-1$. 
	
	Assume from now on that $n\geq 3$. First, $E_i^2 <0$ for all $i$. The contraction of any collection of disjoint $(-1)$-curves permuted transitively by $G$ is $G$-equivariant. Since $(G,X)$ is minimal, there exist $k,l\in \{1,\ldots,n\}$ with $k\neq l$ such that $E_k$ and $E_l$ are two $(-1)$-curves in the same $G$-orbit and $E_k\cap E_l\neq \emptyset$. The image of $E_l$ by the contraction of $E_k$ has self-intersection $0$, and in particular it cannot be contracted. By the assumption that $n\geq 3$, we can contract other $(-1)$-curves in $\kappa^{-1}(p)$, which increases the self-intersection. This contradicts the existence of a contraction of $X$ to the ruled surface $S$, where $f^2=0$ for any fibre $f$. Therefore, we must have $n\leq 2$.
\end{proof}	

Proposition~\ref{Iskovskikh} motivates the study of automorphism groups of conic bundles. The next lemma can be used as a  maximality criterion for their automorphism groups.

\begin{lemma}\label{conicequivariant}
	Let $C$ be a curve of positive genus. Let $\kappa\colon X\to C$ and $\kappa' \colon X'\to C$ be conic bundles. Let $G$ be an algebraic group acting on $X$ and $X'$ such that $(G,X)$ and $(G,X')$ are minimal, and let $\phi\colon X\dashrightarrow X'$ be a $G$-equivariant birational map which is not an isomorphism. Then $\phi=\phi_n \cdots \phi_1$, where each $\phi_j$ is the blowup of a finite $G$-orbit of a point which is contained in the complement of the singular fibres and does not contain two points on the same smooth fibre, followed by the contractions of the strict transforms of the fibres through the points of the $G$-orbit. In particular, $\kappa$ and $\kappa'$ have the same number of singular fibres.
\end{lemma}

\begin{proof}
	Take a minimal resolution of $\phi$, \textit{i.e.}\ a surface $Z$ with $G$-equivariant birational morphisms $\eta\colon Z \to X$ and $\eta'\colon Z\to X'$ satisfying $\eta' = \phi \eta$. Let $E_1,\ldots,E_m\subset Z$ be a $G$-orbit of $(-1)$-curves contracted by $\eta'$, and let $p_i=\kappa\eta(E_i)$. For each $i$, denote by $\widetilde{E}_i$ the image of $E_i$ by $\eta$, which is contained in the fibre $\kappa^{-1}(p_i)$. Since $(G,X)$ is minimal, $\eta$ must blow up a $G$-orbit of points $\Omega$ contained in $\widetilde{E}_1\cup \cdots\cup \widetilde{E}_m$. Hence, $\widetilde{E}_i^2\geq 0$ for all $i$. In particular, $\widetilde{E}_1\cup \cdots\cup \widetilde{E}_m$ is not contained in the set of singular fibres of $\kappa$m and $\widetilde{E}_i^2=0$ for each $i$.

	Then, $\widetilde{E}_1\cup \cdots\cup \widetilde{E}_m$ is contained in the complement of the singular fibres. As $\widetilde{E}_i^2=0$ and $E_i^2=-1$ for each~$i$, no distinct points of $\Omega$ lie in the same smooth fibre. Because $(G,X')$ is minimal, we can contract the strict transforms of the fibres, which yields a $G$-equivariant birational map $\phi_1\colon X\dashrightarrow X_1$ such that $\phi$ factorizes through $\phi_1$. By induction, we find $G$-equivariant birational maps $\phi_j\colon X_{j-1} \dashrightarrow X_j$ such that $\phi = \phi_n \cdots \phi_1$, where each $\phi_j$ is as we wanted.
	
	Finally, applying elementary transformations in the complement of the set of singular fibres does not change the number of singular fibres.  
\end{proof}

\subsection{Generalities on ruled surfaces and their automorphisms}

\begin{definition}\label{defdecomposable}
	A ruled surface $\pi$ is \emph{decomposable} if it admits two disjoint sections. Else, $\pi$ is \emph{indecomposable}.
\end{definition}

The following notion has been already used in \cite{Maruyama2,Maruyama} and more recently in \cite{fong}.

\begin{definition}\label{defSegre}
  The \emph{Segre invariant} $\seg(S)$ of a ruled surface $\pi\colon S \to C$ is the integer $ \min \{\sigma^2\mid \sigma \text{ section of }\pi\}$.
  A section $\sigma$ of $\pi$ such that $\sigma^2 = \seg(S)$ is called a \emph{minimal section}.
\end{definition}

\begin{lemma}\label{segrezero}
	Let $\pi\colon S\to C$ be a decomposable ruled surface with $\seg(S)=0$. If two sections are disjoint, then they are both minimal sections.
\end{lemma}

\begin{proof}
	Let $\sigma$ be a minimal section. Let $\operatorname{Num}(S)$ be the group of divisors of $S$, up to numerical equivalence. Then $\operatorname{Num}(S)$ is generated by the classes of $\sigma$ and $f$, where $\sigma$ is a minimal section and $f$ is a fibre; see \cite[Proposition V.2.3]{Hartshorne}. Let $s_1$ and $s_2$ be disjoint sections. In particular, $s_1 \equiv \sigma + b_1 f$ and $s_2 \equiv \sigma + b_2f$ for some $b_1,b_2\in \mathbb{Z}$. Since we have $s_1\cdot \sigma\geq 0$ and $s_2\cdot \sigma \geq 0$, it follows that $b_1\geq 0$ and $b_2\geq 0$. Since $s_1\cdot s_2 = 0$, this implies that $b_1=b_2=0$.
\end{proof}

\begin{lemma}\label{segrenegative}
	Let $S\to C$ be a ruled surface such that $\seg(S)=-n<0$. The following hold:
	\begin{enumerate}
		\item\label{segreneg-1}  There exists a unique section of negative self-intersection, and all other sections have self-intersection at least~$n$.
		\item\label{segreneg-2} Two sections are disjoint if and only if one is the $(-n)$-section and the other has self-intersection $n$.
		\item\label{segreneg-3} There exists a section of self-intersection $n$ if and only if\, $S$ is decomposable.
	\end{enumerate}
\end{lemma}

\begin{proof}
	\begin{enumerate}[wide]
		\item By assumption, there exists a section $s_{-n}$ of self-intersection $s_{-n}^2=-n<0$. Let $s\neq s_{-n}$ be a section; then $s$ is numerically equivalent to $s_{-n} + bf$ for some integer $b$ and $0\leq s\cdot s_{-n}=s_{-n}^2+b$. Therefore, $b\geq -s_{-n}^2=n$, and it follows that $s^2= s_{-n}^2+2b\geq n$.
		\item Denote by $s_1,s_2$ two disjoint sections. Then $s_1\equiv s_{-n} + b_1f$ and $s_2\equiv s_{-n} + b_2f$ for some $b_1,b_2\in \mathbb{Z}$. We get that $0 = s_1\cdot s_2 = -n + b_1 +b_2$. If $s_1$ and $s_2$ are both different from $s_{-n}$, then $b_1\geq n$ and $b_2\geq n$ by \eqref{segreneg-1}, and this contradicts the equality $0 = -n+b_1+b_2$. Then we can assume that $s_1=s_{-n}$ and $b_1 = 0$. It follows that $b_2=n$ and $s_2^2= s_{-n}^2 + 2n = n$. Conversely, if $s$ is a section of self-intersection $n$, then $s\equiv s_{-n} + nf$ and $s_{-n}\cdot s = 0$.
		\item Let $s$ be a section such that $s^2=n$. Then $s\equiv s_{-n} + nf$ and $s\cdot s_{-n}=0$; \textit{i.e.}\ $s$ and $s_{-n}$ are disjoint sections. In particular, $S$ is decomposable. Conversely, if $S$ is decomposable, there exist two disjoints sections, and one of them has self-intersection $n$ by \eqref{segreneg-2}.
	\end{enumerate}  
\end{proof}

\begin{definition}
	Let $f\in \operatorname{Bir}_C(C\times \PP^1)\simeq \operatorname{PGL}(2,\kk(C))$. The \emph{determinant} of $f$, denoted by $\det(f)$, is the element of $\kk(C)^*/(\kk(C)^*)^2$ defined as the class of the determinant of a representative of $f$ in $\operatorname{GL}(2,\kk(C))$.
\end{definition}

A non-trivial decomposable ruled surface $S$ of Segre invariant $0$ admits exactly two minimal sections. In \cite[Theorem 2(3), (4)]{Maruyama}, a necessary and sufficient condition for such surfaces to have an automorphism permuting two minimal sections is given. We provide below a revisited version of this result, that we prove by computations in local charts. 

\begin{lemma}\label{decomposableseg0}
	Let $C$ be a curve. Let $\pi\colon S=\PP(\mathcal{O}_C(D)\oplus \mathcal{O}_C)\to C$ be a decomposable $\PP^1$-bundle, and let $p_1\colon C\times \PP^1 \to C$ be the trivial $\PP^1$-bundle over $C$. Then $\seg(S)=0$ if and only if $\deg(D)=0$. Moreover, if\, $\seg(S)=0$ and $\pi$ is not trivial, then the following hold:
	\begin{enumerate}
		\item\label{decompseg0-1} The $\PP^1$-bundle $\pi$ has exactly two minimal sections $s_1$ and $s_2$ of self-intersection $0$.
		\item\label{decompseg0-2} We have $\Aut_C(S)\simeq \mathbb{G}_m \rtimes \mathbb{Z}/2\mathbb{Z}$ if\, $2D$ is principal. In this case, for each element $\iota\in \Aut_C(S)$ permuting $s_1$ and $s_2$, there exists a birational map $\xi\colon S\dashrightarrow C\times \PP^1 $ such that $\pi = p_1\xi$ and $\xi \iota \xi^{-1} = 
		\begin{bsmallmatrix}
			0 & \beta \\
			1 & 0
		\end{bsmallmatrix}$, with $\beta\in \kk(C)^*$ and $\operatorname{div}(\beta ) =2D$. In particular, $\iota$ is not a square in $\operatorname{Bir}_C(S)$.
		\item\label{decompseg0-3} We have $\Aut_C(S)\simeq \mathbb{G}_m$ if\, $2D$ is not principal.
	\end{enumerate}
\end{lemma}

\begin{proof}
	We first prove that $\seg(S)=0$ if and only if $\deg(D)=0$. Assume that $\seg(S)=0$. By \cite[Lemma 1.15]{Maruyama2} (see also \cite[Corollary 2.16]{fong}), we get $0 = \deg(\mathcal{O}_C(D)\oplus \mathcal{O}_C) - 2\deg(M)$, where $M$ is a line subbundle of $\mathcal{O}_C(D)\oplus \mathcal{O}_C$ of maximal degree. By the additivity of the degree, $ \deg(D)-2\deg(M)=0$. Since $\deg(M) \geq 0$ and $\deg(M)$ is maximal, $\deg(D) = \deg(M) = 0$. Conversely, we have that $\seg(S)\leq 0$ by \cite[Proposition 2.18(1)]{fong}. Moreover, $S$ admits two disjoint sections corresponding to the line subbundles $\mathcal{O}_C$ and $\mathcal{O}_C(D)$ of $\mathcal{O}_C\oplus \mathcal{O}_C(D)$, and they both have self-intersection $0$ by \cite[Proposition 2.15]{fong}. The case $\seg(S)<0$ is ruled out by Lemma~\ref{segrenegative}\eqref{segreneg-1}, and thus $\seg(S)=0$.
	
	We now assume that $\seg(S)=0$ and that $\pi$ is not trivial, and prove
\eqref{decompseg0-1}, \eqref{decompseg0-2}, \eqref{decompseg0-3}.
 The proof of \eqref{decompseg0-1} can be found in \cite[Lemma 2(2)]{Maruyama} or \cite[Proposition 2.18(3.iii)]{fong}. We now prove \eqref{decompseg0-2} and  \eqref{decompseg0-3}. Let $A$ be a very ample divisor on $C$. For a large enough integer $m$, the divisor $B=D+mA$ is also very ample. In particular, we can find $B' \sim B$ and $A' \sim mA$ such that $\operatorname{Supp}(B')\cap \operatorname{Supp}(D)=\emptyset$ and $\operatorname{Supp}(A')\cap \operatorname{Supp}(D)=\emptyset$. Let $E=A'-B'$; then $D +E \sim 0$, and there exists an $f\in \kk(C)$ such that $\operatorname{div}(f) = D + E$ and $\operatorname{Supp}(D) \cap \operatorname{Supp}(E)=\emptyset$. Choose $U = C\setminus \operatorname{Supp}(E)$ and $V = C\setminus \operatorname{Supp}(D)$ as trivializing open subsets of $\pi$, and local trivializations of $\pi$ such that $s_1$ and $s_2$ are, respectively, the zero and the infinity sections. The transition map of $S$ can be written as
\[s_{vu}\colon \left\{\begin{array}{ccl}
U\times \PP^1 &\xdashrightarrow & V\times \PP^1 \\
(x,[y_0:y_1]) &\longmapsto &(x,[f(x)y_0:y_1]).
\end{array}\right. 
\]
	By \eqref{decompseg0-1}, an element of $\Aut_C(S)$ either fixes pointwise $s_1$ and $s_2$, or  permutes $s_1$ and $s_2$. If $\phi\in \Aut_C(S)$ fixes $s_1$ and $s_2$, then it induces automorphisms $\phi_u \colon U\times \PP^1 \to U\times \PP^1$, $(x,[y_0:y_1])\mapsto (x,[\alpha_u(x) y_0 : y_1])$ and $\phi_v \colon V\times \PP^1 \to V\times \PP^1$, $(x,[y_0:y_1])\mapsto (x,[\alpha_v(x) y_0 : y_1])$ with $\alpha_u\in \mathcal{O}_C(U)^*$, $\alpha_v\in \mathcal{O}_C(V)^*$. The condition $\phi_v s_{vu} = s_{vu} \phi_u$ is equivalent to $\alpha_u = \alpha_v = \alpha \in \mathbb{G}_m$. We have shown that
	$\left\{
	\begin{bsmallmatrix}
		\alpha & 0 \\ 0&1
	\end{bsmallmatrix}\mid\alpha\in \mathbb{G}_m
	\right\}$ is the algebraic subgroup of $\Aut_C(S)$ fixing $s_1$ and $s_2$.
	If $\iota\in \Aut_C(S)$ permutes $s_1$ and $s_2$, then $\iota$ induces automorphisms $\iota_u \colon U\times \PP^1 \to U\times \PP^1$, $(x,[y_0:y_1])\mapsto (x,[\beta_u(x) y_1 : y_0])$ and $\iota_v \colon V\times \PP^1 \to V\times \PP^1$, $(x,[y_0:y_1])\mapsto (x,[\beta_v(x) y_1 : y_0])$ with $\beta_u\in \mathcal{O}_C(U)^*$, $\beta_v\in \mathcal{O}_C(V)^*$. The condition $\iota_v s_{vu} = s_{vu}\iota_u$ is now equivalent to $f^2 \beta_u = \beta_v$. In particular, $\operatorname{div}(\beta_v) = 2D$, and $2D$ is a principal divisor. Conversely, if $2D$ is a principal divisor, there exists a $\beta\in \kk(C)^*$ such that $\operatorname{div}(\beta) = 2D$. Choose $\beta_v = \beta$ and $\beta_u = f^{-2} \beta_v$; the automorphisms $\iota_u$, $\iota_v$ glue back to a $C$-automorphism $\iota$ of $S$ of order two which permutes $s_1$ and $s_2$. Thus, $\Aut_C(S)\simeq \mathbb{G}_m \rtimes \mathbb{Z}/2\mathbb{Z}$ if and only if $2D$ is principal, and $\iota_v$ induces the birational map $\xi\colon S\dashrightarrow C\times \PP^1$ given in the statement.
	
	Finally, assume $\iota$ is a square in $\operatorname{Bir}(S)$. Then $\xi \iota \xi^{-1}$ is a square in $\operatorname{Bir}(C\times \PP^1)$ and has determinant $-\beta$. Since $\pi$ is not trivial by assumption, it follows that $D$ is not principal and $\operatorname{div}(\beta)=2D$. This implies that $-\beta=\det(\xi \iota \xi^{-1})$ is not a square, which gives a contradiction.
\end{proof}

Every element of $\Aut_{\PP^1}(\mathbb{F}_n)$ fixes pointwise a section of $\FF_n$. This is not true when we consider $\PP^1$-bundles over a non-rational curve $C$, as we have seen in Lemma~\ref{decomposableseg0}\eqref{decompseg0-2}. The following lemma shows that it is the only exception up to conjugation. 

\begin{lemma}\label{invariantsection}
	Let $C$ be a curve. Let $\pi\colon S\to C$ be a ruled surface, and let $p_1\colon C\times \PP^1 \to C$ be the trivial $\PP^1$-bundle and $f\in \Aut_C(S)$. Then $f$ satisfies one of the following:
	\begin{enumerate}
		\item\label{invsec-1} The morphism $f$ fixes pointwise a section of\, $\pi$.
		\item\label{invsec-2} The morphism $f$ does not fix any section of\, $\pi$, and there exists a birational map $\xi\colon S\dashrightarrow C\times \PP^1 $ such that $\pi=p_1\xi$ and $\xi f \xi^{-1} = 
		\begin{bsmallmatrix}
			0 & \beta \\
			1 & 0
		\end{bsmallmatrix}$, with $\operatorname{div}(\beta ) =2D$ for some divisor $D$ which is not principal.
	\end{enumerate}
	Moreover, if $f$ satisfies \eqref{invsec-2}, then $f$ is not a square in $\operatorname{Bir}_C(S)$.
\end{lemma}

\begin{proof}
	First we deal with the case $\seg(S)\leq 0$. If $\seg(S)<0$, or $\seg(S)=0$ and $S$ is indecomposable, then $S$ has a unique minimal section which is $\Aut_C(S)$-invariant (see Lemma~\ref{segrenegative}\eqref{segreneg-1} and \cite[Lemma 2(1.ii)]{Maruyama}, or \cite[Proposition 2.18(3.ii)]{fong}). If $S$ is trivial, then $\Aut_C(S) = \operatorname{PGL}(2,\kk)$ and every element fixes pointwise a section. In particular, $f$ satisfies  condition \eqref{invsec-1}. Else $\seg(S)=0$, $S$ is decomposable, and $S$ is not trivial. Then $S = \mathbb{P}(\mathcal{O}_C(D)\oplus \mathcal{O}_C)$ for some divisor $D$ of degree $0$, and by Lemma~\ref{decomposableseg0}, $\Aut_C(S)\simeq \mathbb{G}_m$ or $\mathbb{G}_m\rtimes \mathbb{Z}/2\mathbb{Z}$. In particular, the automorphism $f$ either fixes the two minimal sections of $S$ and satisfies \eqref{invsec-1}, or permutes them and satisfies \eqref{invsec-2}.
	
	Assume $\seg(S)>0$. Then $S$ is indecomposable, and $\Aut_C(S)$ is finite; see \cite[Theorem 2(1)]{Maruyama}. Let $s$ be a section of $S$. If $f(s) = s$, we are done. Else $s$ and $f(s)$ intersect in finitely many points which are fixed by~$f$. Blow up these points and contract the strict transforms of their fibres, and repeat the process until  the strict transforms of the sections are disjoint. This yields an $f$-equivariant birational map $\phi\colon  S\dashrightarrow  S'$ with $S'$ decomposable. By \cite[Proposition 2.18(1)]{fong}, it follows that $\seg(S')\leq 0$. Moreover, the strict transforms of $s$ and $f(s)$ by $\phi$ are disjoint and permuted by $\phi f \phi^{-1}$; hence  $\seg(S')=0$ (see Lemma~\ref{segrenegative}\eqref{segreneg-2}). Then Lemma~\ref{decomposableseg0} implies that $ f$ satisfies \eqref{invsec-2}. Since $\phi f \phi^{-1}$ is a not a square in $\operatorname{Bir}_C(S)$ by Lemma~\ref{decomposableseg0}\eqref{decompseg0-2}, it also follows that $f$ is also not a square in $\operatorname{Bir}_C(S)$.
\end{proof}
 
\subsection{Reduction of cases}

The following lemma is an analogue of \cite[Lemma 6.1$(1)\Leftrightarrow (2)$]{Blanc2} for not necessarily rational conic bundles. The proof is slightly more difficult, due to  case \eqref{invsec-2} of Lemma~\ref{invariantsection}, which does not exist in the rational case.

\begin{lemma}\label{twisting}
	Let $C$ be a curve. Let $\kappa\colon X\to C$ be a conic bundle with at least one singular fibre, and let $f\in \Aut_C(X)$ permute the irreducible components of at least one singular fibre. Then $f$ has order two.
\end{lemma}

\begin{proof}
	Let $\eta\colon X\to S$ be the contraction of one irreducible component in each singular fibre. The automorphism $f^2$ preserves all the irreducible components of the singular fibres; hence $\eta$ is $f^2$-equivariant. Let $g = \eta f^2 \eta^{-1}\in \Aut_C(S)$, which is a square in $\operatorname{Bir}_C(S)$; then $g$ fixes pointwise a section (see Lemma~\ref{invariantsection}). Let $s_{\inv}'$ be a $g$-invariant section of $S$, and $s_{\inv}$ its strict transform by $\eta$ which is $f^2$-invariant. As $f$ exchanges the irreducible components of at least one singular fibre, the section $s_{\inv}$ is not $f$-invariant. The sections $s_{\inv}$ and $f(s_{\inv})$ meet a general fibre in two points which are exchanged by the action of $f$. Thus $f$ has order two. 
\end{proof}

\begin{lemma}\label{invariantexceptional}
	Let $C$ be a curve. Let $\kappa\colon X\to C$ be a conic bundle with at least one singular fibre, such that its two irreducible components are exchanged by an element $\rho\in \Aut(X)$. Let $G$ be the normal subgroup of $\Aut_C(X)$ which leaves invariant each irreducible component of the singular fibres. The following hold:
	\begin{enumerate}
		\item\label{invariantexc-1} If $G$ fixes a section $\sigma_1$ of\, $\kappa$ and $G$ is not trivial, then $\kappa$ is an exceptional conic bundle.
		\item\label{invariantexc-2} If there exists a contraction $\eta\colon X\to S$ such that $\seg(S)\leq 0$ and $S$ is indecomposable, then $G$ is trivial.
	\end{enumerate}
\end{lemma}

\begin{proof}
	\begin{enumerate}[wide]
		\item The subgroup $G\subset \Aut(X)$ is normal; hence $\rho G \rho^{-1} =G$, and the section $\sigma_2=\rho \sigma_1\neq \sigma_1$ is also $G$-invariant. Let $\eta\colon X\to S$ be the contraction of one irreducible component in each singular fibre of $\kappa$, namely the one intersecting $\sigma_2$; then it is a $G$-equivariant birational morphism. Let $H=\eta G \eta^{-1}\subset \Aut_C(S)$, which is not trivial. The images of $\sigma_1$ and $\sigma_2$ by $\eta$ are $H$-invariant sections $s_1$ and $s_2$ of $S$. Assume that $s_1$ and $s_2$ intersect. Choose another section $s_3$. Apply elementary transformations centred on $\{s_i\cap s_j\mid i,j\in \{1,2,3\},i\neq j\}$, and repeat until  the strict transforms of $s_1,s_2,s_3$ are disjoint. This yields an $H$-equivariant birational map $\psi\colon S\dashrightarrow C\times \PP^1$. The group $\psi H \psi^{-1}$ is an algebraic subgroup of $\operatorname{PGL}(2,\kk)$ which fixes the strict transforms of $s_1$, $s_2$ and the basepoints of $\psi^{-1}$. The basepoints of $\psi^{-1}$ coming from the contraction of the strict transforms of the fibres passing through the intersections of $s_1$ and $s_2$ are outside of the strict transforms of $s_1$ and $s_2$. Then $H$ is conjugate to a subgroup of $\PGL(2,\kk)$ fixing three distinct points on $\PP^1$, which implies that $H$ is trivial and gives a contradiction. Therefore, $s_1$ and $s_2$ are disjoint sections of $S$, and it follows that $S$ is decomposable and $\seg(S)\leq 0$ by \cite[Proposition 2.18(1)]{fong}. Thus $\sigma_1$ and $\sigma_2$ are also disjoint sections of $X$ which pass through different irreducible components in each singular fibre. 
		
		Since $\sigma_2 = \rho \sigma_1$, it follows that $\sigma_1^2 = \sigma_2^2$. Then $s_1^2<s_2^2$. If $\seg(S)=0$, then $s_1$ and $s_2$ are both minimal sections as they are disjoint, and this contradicts the inequality $s_1^2<s_2^2$. Therefore, $\seg(S)<0$ and by Lemma~\ref{segrenegative}\eqref{segreneg-2}, it follows that $s_1^2 = -n<0$ and $s_2^2 = n>0$ for $n=-\seg(S)$. In particular, $\eta$ is the blowup of $2n$ points on $\sigma_2$. Then $\kappa$ has $2n$ singular fibres and two disjoint $(-n)$-sections; \textit{i.e.}\ it is an exceptional conic bundle.
		\item If $\seg(S)\leq 0$ and $\pi$ is indecomposable, then $S$ has a unique minimal section which is $\Aut(S)$-invariant, see \cite[Lemma 2(1)(i), (ii)]{Maruyama} (or \cite[Proposition 2.18(2), (3.ii)]{fong}), and its strict transform by $\eta$ is a $G$-invariant section of $\kappa$. If $G$ is not trivial, it follows from \eqref{invariantexc-1} that $X$ is an exceptional conic bundle. This implies that $S$ admits two disjoint sections, which gives a contradiction.\hfill\qedhere	\end{enumerate}
\end{proof}

The key result of this section is the following proposition, analogue of \cite[Lemma 4.3.5]{Blanc}, which will be useful to reduce to the study of automorphism groups of ruled surfaces, exceptional conic bundles and $(\mathbb{Z}/2\mathbb{Z})^2$-conic bundles.

\begin{proposition}\label{reduction}
	Assume that $\operatorname{char}(\kk)\neq2$. Let $C$ be a curve. Let $\kappa\colon X\to C$ be a conic bundle with at least one singular fibre, such that its two irreducible components are exchanged by an element of $\Aut(X)$. Let $G$ be the normal subgroup of $\Aut_C(X)$ which leaves invariant every irreducible component of the singular fibres. If $G$ is not trivial and if there exists a contraction $\eta\colon X\to S$ with $S$ a decomposable $\PP^1$-bundle over $C$, then $\kappa$ is an exceptional conic bundle. Else, $\Aut_C(X)$ is isomorphic to $(\mathbb{Z}/2\mathbb{Z})^r$ for some $r\in \{0,1,2\}$. 
\end{proposition}

\begin{proof}
	If $G$ is trivial, then every element of $\Aut_C(X)$ is an involution. This implies that $\Aut_C(X)$ is a finite subgroup of $\operatorname{PGL}(2,\kk(C))$, and the statement follows. 
	Assume that $G$ is not trivial, and let $\eta\colon X \to S$ be a contraction, where $\pi\colon S\to C$ is a $\PP^1$-bundle. Then $\eta$ is $G$-equivariant, and $H=\eta G \eta^{-1}\subset \Aut_C(S)$ is not trivial. Three cases arise:
	\begin{enumerate}[wide]
		\item First assume  that $\seg(S)< 0$. Then $S$ admits a unique minimal section, and its strict transform by $\eta$ is a $G$-invariant section of $\kappa$. By Lemma~\ref{invariantexceptional}, $S$ is decomposable, and $\kappa$ is an exceptional conic bundle.
		 
		 \item Assume that $\seg(S)=0$. If a section of $S$ of self-intersection $0$ passes through at least one of the points blown up by $\eta$, its strict transform is a section $s$ of $X$ of negative self-intersection. Contracting in each fibre the irreducible component not intersecting $s$ gives a birational morphism $\eta'\colon X\to S'$ with $\seg(S')<0$, reducing to the previous case. We now assume that no section of $S$ of self-intersection $0$ passes through any point blown up by $\eta$. Firstly, $\pi\colon S\to C$ is not a trivial bundle, as otherwise sections of $S$ of self-intersection $0$ would cover $S$.  From Lemma~\ref{invariantexceptional}\eqref{invariantexc-2}, $S$ is decomposable. Moreover, by Lemma~\ref{decomposableseg0}\eqref{decompseg0-1}, there exist exactly two disjoint sections $s_1',s_2'$ of $S$ of self-intersection $0$. Furthermore, $\eta G\eta^{-1}$ is a non-trivial subgroup of  $\Aut_C(S)$, isomorphic to $\mathbb{G}_m$ or $\mathbb{G}_m\rtimes \mathbb{Z}/2\mathbb{Z}$ (see Lemma~\ref{decomposableseg0}), that fixes the basepoints of $\eta^{-1}$ not lying on $s_1'$ or $s_2'$. We now prove that no non-trivial element of $\eta G \eta^{-1}$ can lie in $\mathbb{G}_m$: take a trivializing open subset $U\subseteq C$ of $\pi\colon S\to C$ containing the image of a basepoint, and take an isomorphism $\pi^{-1}(U)\simeq U\times \mathbb{P}^1$ sending $s_1',s_2'$ onto the zero and infinity sections; then the action of $\mathbb{G}_m$ on $U\times \mathbb{P}^1$ is  $(x,[u:v])\mapsto (x,[\alpha u:v])$, and thus no non-trivial element of $\mathbb{G}_m$ fixes any point outside of $s_1',s_2'$. Then $\eta G \eta^{-1} \cap \mathbb{G}_m =\{1\}$, and $G$ has order two. By Lemma~\ref{twisting}, every element of $\Aut_C(X)$ is an involution. As $\Aut_C(X)$ is a finite subgroup of $\PGL(2,\kk(C))$, this implies that $\Aut_C(X)\simeq (\mathbb{Z}/2\mathbb{Z})^r$ for some $r\in \{0,1,2\}$.
		 
		\item Assume that $\seg(S)>0$. In particular, $S$ is indecomposable (see \cite[Proposition 2.18(1)]{fong}). Then from \cite[Lemma 3]{Maruyama}, $\Aut_C(S)$ is isomorphic to a subgroup of $\Pic^0(C)[2]$. In particular, it is a finite subgroup of $\operatorname{PGL}(2,\kk(C))$ such that every element is an involution. Hence $\Aut_C(S) \simeq (\mathbb{Z}/2\mathbb{Z})^s$ for some $s\in \{0,1,2\}$. It follows that every element of $G$ is an involution, and by Lemma~\ref{twisting} every element of $\Aut_C(X)$ is an involution. Since $\Aut_C(X)$ is a finite subgroup of $\operatorname{PGL}(2,\kk(C))$, it follows that $\Aut_C(X)\simeq (\mathbb{Z}/2\mathbb{Z})^r$ for some $r\in \{0,1,2\}$. 
\hfill\qedhere	\end{enumerate}
\end{proof}

\section{Automorphism groups of irrational conic bundles}

\subsection{Infinite increasing sequence of automorphism groups}

We first prove the following lemma, which is a generalization of \cite[Theorem A]{fong}. The proof works essentially the same, based on an explicit automorphism of ruled surfaces computed in \cite{Maruyama}.

\begin{lemma}\label{infiniteseq}
	Let $C$ be a curve of positive genus and $\pi\colon S\to C$ be a ruled surface such that $\seg(S)<0$. Then there exists an infinite family $\{S_i,\phi_i\}_{i\geq 1}$, where the $\pi_i\colon S_i\to C$ are ruled surfaces and the $\phi_i\colon S \dashrightarrow S_i$ are $\Aut(S)$-equivariant birational maps, such that
	$$
	\Aut(S) \subsetneq \phi_1^{-1} \Aut(S_1) \phi_1 \subsetneq \cdots \subsetneq \phi_n^{-1} \Aut(S_n) \phi_n \subsetneq \cdots
	$$
	is an infinite increasing sequence of algebraic subgroups of\, $\operatorname{Bir}(C\times \PP^1)$. In particular, $\Aut(S)$ is not a maximal algebraic subgroup of $\operatorname{Bir}(C\times \PP^1)$.
\end{lemma}

\begin{proof}
	Since $\seg(S)<0$, there exists a unique negative section (see Lemma~\ref{segrenegative}\eqref{segreneg-1}), which is $\Aut(S)$-invariant. From \cite[Lemmas 6 and 7]{Maruyama}, the morphism of algebraic groups $\pi_*\colon \Aut(S)\to \Aut(C)$ has finite image. Let $p$ be a point on the minimal section; its orbit by the $\Aut(S)$-action is a finite subset of the minimal section. The blowup of the orbit of $p$ followed by the contractions of the strict transforms of the fibres defines an $\Aut(S)$-equivariant birational map $\eta_1 \colon S\dashrightarrow S_1$ with $\seg(S_1) < \seg(S) $. Repeating this process gives rise to a family of ruled surfaces $\{\pi_i\colon S_i\to C \}_{i\geq 1}$ with an infinite sequence
	\begin{equation}
		\Aut(S) \subset \phi_1^{-1} \Aut(S_1) \phi_1 \subset \cdots \subset\phi_n^{-1} \Aut(S_n) \phi_n \subset \cdots,\label{notstationaryseq} \tag{$\dagger$}
	\end{equation}
	where $\phi_i = \eta_i \cdots \eta_1$. We will see that this sequence is not stationary.
	
	Take $n$ large, and let $z=\pi(p)$. By a choice of trivialization $\pi_n^{-1}(U)\simeq U\times \PP^1$, we can assume that $q=(z,[0:1])\in U\times \PP^1$ is a basepoint of $\eta_n^{-1}\colon S_n \dashrightarrow S_{n-1}$. Let $V$ be a vector bundle of rank two over~$C$ such that $\PP(V)=S_n$, and let $L\subset V$ be the line subbundle associated to the minimal section in $S_n$. Let $\mathcal{L}=\det(V)^{-1}\otimes L^2$; it follows from \cite[Corollary 2.16]{fong} that $\deg(\mathcal{L})= -\seg(S_n)$. Since $n$ is chosen large, we can assume that $\seg(S_n)<0$ is small enough such that $\mathrm{h}^1(C,\mathcal{L}) = \mathrm{h}^1(C,\mathcal{L}\otimes \mathcal{O}_C(z)^{-1})=0$. By the Riemann--Roch theorem, we get that $h^0(C,\mathcal{L}\otimes \mathcal{O}_C(z)^{-1}) < h^0(C,\mathcal{L})$; \textit{i.e.}\ $z$ is not a basepoint of the complete linear system $|\mathcal{L}|$. Therefore, there exists a $\gamma \in \mathrm{H}^0(C,\mathcal{L})$ such that $\gamma(z)\neq 0$. 
	
	Let $(U_i)_i$ be trivializing open subsets of $\pi_n$; the automorphisms
	\begin{align*}
		U_i \times \PP^1 & \longrightarrow U_i\times \PP^1 \\
		(x,[y_0:y_1]) & \longmapsto (x,[y_0+y_1\gamma_{|U_i}(x):y_1])
	\end{align*}
	glue into a $C$-automorphism $f_\gamma$ of $S_n$ (see \cite[case (b), p.~92]{Maruyama}) such that $f_\gamma$ does not fix $q$ and $\Aut(S_{n-1}) \subsetneq \eta_n^{-1}\Aut(S_n)\eta_n$. We have proved that the sequence \eqref{notstationaryseq} is not stationary. Removing in the sequence the groups which are not strictly bigger than the previous term and renaming the elements accordingly yields the increasing sequence of the statement.
\end{proof}

\begin{remark}\label{dimensionarbitrarylarge}
	Notice that the proof of Lemma~\ref{infiniteseq} implies \cite[Theorem A]{fong}. Let $\gamma \in \Gamma(C,\det(V)^{-1}\otimes L^2)$ be as above. For any $t\in \mathbb{G}_a$, the automorphisms 
	\begin{align*}
		U_i \times \PP^1 & \longrightarrow U_i\times \PP^1 \\
		(x,[y_0:y_1]) & \longmapsto (x,[y_0+y_1t\gamma_{|U_i}(x):y_1])
	\end{align*}
	glue into a $C$-automorphism $f_{t\gamma}$. In particular, each automorphism $f_\gamma$ belongs to the connected component of the identity. Restricting the infinite chain \eqref{notstationaryseq} to the connected components, one gets that
	$$
	\Aut^\circ(S) \subsetneq \phi_1^{-1} \Aut^\circ(S_1) \phi_1 \subsetneq \cdots \subsetneq\phi_n^{-1} \Aut^\circ(S_n) \phi_n \subsetneq \cdots,
	$$
	and in particular $\dim(\Autzero(S_n))<\dim(\Autzero(S_{n+1}))$ for all $n$.
\end{remark}

\subsection{Exceptional conic bundles}

The following lemma is a generalization of \cite[Lemma 4.3.1]{Blanc} for exceptional conic bundles which are not necessarily rational.

\begin{lemma}\label{exceptional}
	Let $C$ be a curve, and let $\kappa\colon X\to C$ be a conic bundle with $2n\geq 0$ singular fibres. The following assertions are equivalent:
	\begin{enumerate}
		\item\label{exc-1} The bundle $\pi$ is exceptional.
		\item\label{exc-2} There exist exactly two sections $s_1,s_2$ of self-intersection $-n$, which are disjoint and intersect different irreducible components of each singular fibre.
		\item\label{exc-3} There exists a birational morphism $\eta_n\colon X\to S$, where $S$ is a decomposable $\PP^1$-bundle over $C$ with $\seg(S)=-n$, which consists in the blowup of\, $2n$ points on a section of self-intersection $n$ in $S$.
		\item\label{exc-4} There exists a birational morphism $\eta_0\colon X \to S$, where $S$ is a decomposable $\PP^1$-bundle over $C$ with $\seg(S)=0$, which consists in the blowup of\, $2n$ points such that no two points are in the same fibre, $n$ are chosen on a section of self-intersection $0$, and the other $n$ are chosen on another section of self-intersection $0$.
	\end{enumerate}
\end{lemma}
	
\begin{proof}
	\begin{enumerate}[wide]
		\item[\eqref{exc-1} $\implies$ \eqref{exc-2}, \eqref{exc-3}] Assume $\kappa$ is exceptional, and let $s_1,s_2$ be sections of self-intersection $-n$. Contracting in each singular fibre the irreducible component which does not meet $s_1$ yields a birational morphism $\eta_n\colon X \to S$, where $S$ is a ruled surface over $C$. Denote by $s_1'$ and $s_2'$ the images of $s_1$ and $s_2$ by $\eta_n$; then $s_1'^2 = -n$ and $s_2'^2 \leq n$. The case $s_2'^2 < n$ cannot happen (see Lemma~\ref{segrenegative}\eqref{segreneg-1}), and the equality implies $s_1$ and $s_2$ pass through different irreducible components of each singular fibre. Then the sections $s_1$ and $s_2$ are disjoint (see Lemma~\ref{segrenegative}\eqref{segreneg-2}), and $S$ is decomposable (see Lemma~\ref{segrenegative}\eqref{segreneg-3}). Assume there exists a third section $s_3$ of self-intersection $-n$ on $X'$. By the same argument, $s_3$ has to pass through different irreducible components than $s_1$ and $s_2$. Since each singular fibre contains exactly two irreducible components, this is not possible.
		\item[\eqref{exc-2} $\implies$ \eqref{exc-4}]  Contract in $n$ singular fibres the irreducible components meeting $s_1$, and contract in the other singular fibres the irreducible components meeting $s_2$. This defines a birational morphism $\eta_0\colon X \to S$ such that the images of $s_1$ and $s_2$ by $\eta_0$ are disjoint sections of $S$ of self-intersection $0$. In particular, $S$ is decomposable, and by Lemma~\ref{segrenegative}\eqref{segreneg-1}, $\seg(S)=0$.
		\item [\eqref{exc-2}, \eqref{exc-3}, \eqref{exc-4} $\implies$ \eqref{exc-1}] The implication \eqref{exc-2} $\implies$ \eqref{exc-1} is trivial. The strict transforms by $\eta_n$ of the sections of $S$ of self-intersection $n$ and $-n$ are two sections of $\kappa$ of self-intersection $-n$. The strict transforms by $\eta_0$ of the two sections of $S$ of self-intersection $0$ are two sections of $\kappa$ of self-intersection $-n$. This proves that \eqref{exc-3} and \eqref{exc-4} imply  \eqref{exc-1}.
	\end{enumerate}
\end{proof}

In \cite[Lemma 4.3.3(1)]{Blanc}, it is proven that $\Aut_{\PP^1}(X)\simeq \mathbb{G}_m \rtimes \mathbb{Z}/2\mathbb{Z}$ when $X$ is an exceptional conic bundle over $\PP^1$, which implies that $\Aut(X)$ is maximal. We see below that automorphism groups of exceptional conic bundles over a non-rational curve do not always contain an involution permuting the two $(-n)$-sections (see Proposition~\ref{exceptionalinvolution}) and are not always maximal (see Lemma~\ref{exceptionalmaximal}).

\begin{lemma}\label{exceptionalmaximal}
	Let $C$ be a curve of positive genus, and let $\kappa\colon X\to C$ be an exceptional conic bundle. If $\Aut_C(X)$ contains a non-trivial involution permuting the irreducible components of the singular fibres, then $\Aut_C(X)\simeq\mathbb{G}_m \rtimes \mathbb{Z}/2\mathbb{Z}$ and $\Aut(X)$ is maximal. Else, $\Aut_C(X)\simeq \mathbb{G}_m$, and $\Aut(X)$ can be embedded in a infinite increasing sequence of algebraic subgroups of\, $\Bir(C\times \PP^1)$. 
\end{lemma}

\begin{proof}
	Denote by $s_1,s_2$ the two $(-n)$-sections of $\kappa$. Let $G$ be the subgroup of $\Aut_C(X)$ which leaves invariant the irreducible components of each singular fibre, and let $\eta_0\colon X \to S$ be a birational morphism which contracts an irreducible component in each singular fibre and is such that $\seg(S)=0$ (see Lemma~\ref{exceptional}\eqref{exc-4}). Let $\pi\colon S\to C$ be the morphism such that $\kappa = \pi \eta_0$. Let $s_1'$, $s_2'$ be, respectively, the images of $s_1$ and $s_2$ by~$\eta_0$. Then $\eta_0$ is $G$-equivariant, and $\eta_0 G \eta_0^{-1}$ is an algebraic subgroup of $\Aut_C(S)$ which leaves invariant $s_1'$ and~$s_2'$. If $\pi$ is trivial, then $\eta_0 G \eta_0^{-1}$ is an algebraic subgroup of $\PGL(2,\kk)$ fixing at least two points on a fibre, and thus is contained in $\mathbb{G}_m \subset \operatorname{PGL}(2,\kk)$. If $\pi$ is not trivial, then $\eta_0 G \eta_0^{-1}$ is an algebraic subgroup of $\mathbb{G}_m$ or  $\mathbb{G}_m\rtimes \mathbb{Z}/2\mathbb{Z}$ by Lemma~\ref{decomposableseg0}. Since the element of order two in $\mathbb{Z}/2\mathbb{Z}$ permutes $s_1'$ and $s_2'$, it follows that $\eta_0 G \eta_0^{-1}$ is also contained in $\mathbb{G}_m$ in the second case. Hence $\eta_0 G \eta_0^{-1}$ is contained in a subgroup of $\Aut_C(S)$ isomorphic to $\mathbb{G}_m$. Conversely, every element of $\mathbb{G}_m$ fixes $s_1',s_2'$ (the images of $s_1,s_2$ by $\eta_0$); hence $\eta_0 G\eta_0^{-1} = \mathbb{G}_m$.
	
	The exceptional conic bundle $\kappa$ has exactly two $(-n)$-sections, which are left invariant or are permuted by the elements of $\Aut(X)$. Assume that $\Aut_C(X)$ contains an element $\iota$ which permutes the two $(-n)$-sections of $\kappa$ (or, equivalently, the irreducible components of each singular fibre, by Lemma~\ref{exceptional}\eqref{exc-2}). The automorphism $\iota$ acts on a general fibre by permuting two points, which implies that $\iota$ is an involution. If $f\in \Aut_C(X)$ permutes the two $(-n)$-sections, then $\iota f$ does not; \textit{i.e.} $\iota f\in G$. This implies that $\Aut_C(X) = G\rtimes \langle \iota\rangle$. Moreover, there is no $\iota$-equivariant contraction from $X$, and all $\Aut(X)$-orbits in the complement of the singular fibres are infinite or contain two points on a smooth fibre. Hence there is no $\Aut(X)$-equivariant birational map from $X$ (see Lemma~\ref{conicequivariant}), and $\Aut(X)$ is maximal.
	
	If  $\Aut_C(X)$ does not contain an element $\iota$ as above, then  $\Aut_C(X) = G \simeq \mathbb{G}_m$. Since there exist exactly two $(-n)$-sections, an element of $\Aut(X)$ either fixes them or permutes them. The contraction of $\Aut(X)$-orbits of $(-1)$-curves yields
	an $\Aut(X)$-equivariant birational morphism $\eta_n\colon X \to S$, where $\seg(S)<0$ (see Lemma~\ref{exceptional}\eqref{exc-3}). Then use Lemma~\ref{infiniteseq} to conclude.
\end{proof}

\begin{proposition}\label{exceptionalinvolution}
	Let $C$ be a curve. Let $\kappa\colon X\to C$ be a conic bundle with two disjoint sections $s_1$ and $s_2$ passing through different irreducible components of each singular fibre. Let $\eta\colon X\to S$ be the contraction of an irreducible component in each singular fibre. This yields a ruled surface $\pi\colon S\to C$ such that $\kappa = \pi \eta$. Denote by $s_1'$ and $s_2'$ the images of $s_1$ and $s_2$ by $\eta$. The following hold:
	\begin{enumerate}
	\item\label{excinv-1} We have that $\pi$
          is decomposable; \textit{i.e.}\ there exists a $D\in \Pic(C)$ such that $S=\PP(\mathcal{O}_C(D)\oplus \mathcal{O}_C)$.
		\item\label{excinv-2} The birational morphism $\eta$ is the blowup of finite sets $Z\subset s_1'$ and $P\subset s_2'$.
		\item\label{excinv-3} The group $\Aut_C(X)$ contains a non-trivial involution permuting the irreducible components of each singular fibre if and only if the divisor $-2D$ is linearly equivalent to
		$$
		\sum_{z\in \pi(Z)} z - \sum_{p\in \pi(P)} p.
		$$
		\item\label{excinv-4} If one of the conditions of \eqref{excinv-3} holds, then $\kappa$ is an exceptional conic bundle.
	\end{enumerate}
\end{proposition}

\begin{proof}
	\begin{enumerate}[wide]
		\item Each fibre of $\pi$ is isomorphic to $\PP^1$. It follows that $\pi$ is a ruled surface, which is decomposable because $s_1'$ and $s_2'$ are disjoint. 
		\item Each irreducible component of a singular fibre intersects $s_1$ or $s_2$. It follows that $\eta$ is the blowup of finitely many points lying in $s_1'$ or $s_2'$.
		\item Up to a choice on the trivialization of $S$, we can also assume that $s_1'$ is the zero section and $s_2'$ is the infinity section. Replacing $D$ with another divisor of its linear class, we can assume that $\operatorname{Supp}(D) \cap \pi(Z\cup P) = \emptyset$. Let $U\subset C$ be a trivializing open subset of $\pi$ containing $\pi(Z\cup P)$ and such that $\operatorname{Supp}(D)\subset C\setminus U$. 
		
		First assume that there exists a $f\in \kk(C)$ such that 
		$${\operatorname{div}(f)} = \sum_{z\in \pi(Z)} z - \sum_{p\in \pi(P)} p +2D.$$
		Then define the birational map $\phi_u\colon U\times \PP^1 \dashrightarrow U\times \PP^1$, $(x,[y_0:y_1])\mapsto (x,[f(x)y_1:y_0])$, which is involutive and has basepoints at $Z\cup P$. Take another trivializing open subset $V$ with a trivialization map such that the transition function of $\pi$ equals
		$g_{uv}\colon V\times \PP^1 \dashrightarrow U\times \PP^1$, $(x,[y_0:y_1])\mapsto (x,[\alpha_{uv}(x) y_0:y_1])$, where $\alpha_{uv}\in \kk(C)^*$ denotes the transition  function of $\mathcal{O}_C(D)$. Denote by $\nu_q$ the multiplicity at $q\in C$; then  $\nu_q(\alpha_{uv}^{-2}f) = \nu_q(f)-2\nu_q(\alpha_{uv} )=0$ for all $q\in V\setminus U$. This implies that $\phi_v = g_{uv}^{-1} \phi_u g_{uv}\colon (x,[y_0:y_1]) \mapsto (x,[\alpha_{uv}^{-2}(x)f(x)y_1:y_0])$ extends to a birational map defined on $(V\setminus U)\times \PP^1$. Hence $\phi_u$ extends to a $C$-birational map $\phi$ of $S$, and $\eta^{-1} \phi \eta\in \Aut_C(X)$ is an involution permuting the irreducible components of the singular fibres of $\kappa$. 
		
		Conversely, assume there exists an involution $\psi \in \Aut_C(X)$ permuting the irreducible components of the singular fibres. Then $\eta \psi \eta^{-1}\in \operatorname{Bir}(S)$ acts trivially on $C$ and permutes $s_1$ and $s_2$; hence there exists an $f\in \kk(C)$ such that the restriction of $\eta \psi \eta^{-1}$ to $\pi^{-1}(U)$ yields a birational map 
\[
\phi_u\colon \left\{\begin{array}{rcl}
U\times \PP^1 & \xdashrightarrow & U\times \PP^1 \\
			(x,[y_0:y_1]) & \longmapsto &(x,[f(x)y_1:y_0]).
\end{array} \right.
\]
		Since the set of basepoints of $\eta \psi \eta^{-1}$ is exactly $Z\cup P$, the rational function $f_{|U}$ has zeros in $\pi(Z)$ and poles in $\pi(P)$. Computing in local charts, one can check that $\nu_q(f) = 1$ if $q\in \pi(Z)$ and $\nu_q(f) = -1 $ if $q\in \pi(P)$. Conjugating as before by the transition maps of $\pi$ gives $\phi_v=g_{uv}^{-1}\phi_u g_{uv}\colon V\times \PP^1 \dashrightarrow V\times \PP^1$, $(x,[y_0:y_1]) \mapsto (x,[\alpha_{uv}^{-2}(x)f(x)y_1:y_0])$. The birational map $\eta \psi \eta^{-1}$ is biregular on $\pi^{-1}(C\setminus U)$; hence $\nu_q(\alpha_{uv}^{-2}f) = 0$ for all $q\in V\setminus U$. Thus $\operatorname{div}(f)=\sum_{z\in \pi(Z)} z - \sum_{p\in \pi(P)} p +2D$.
		
		\item Let $m$ be the number of singular fibres of $\kappa$, and denote by $\iota\in \Aut_C(X)$ a non-trivial involution permuting the irreducible components of each singular fibre. Without loss of generality, we can replace $s_2$ with $\iota(s_1)$, and it follows that $s_1^2= s_2^2$. Contracting the irreducible components intersecting $s_1$ in each singular fibre of $\kappa$ gives a birational morphism $\eta\colon X \to S$, where $S$ is decomposable with two disjoint sections $s_1'$ and~$s_2'$. Then $\seg(S)\leq 0$ (see \cite[Corollary 1.17]{Maruyama2} or \cite[Proposition 2.18(1)]{fong}), and $\seg(S)\neq 0$ by the definition of $\eta$ and by Lemma~\ref{segrezero} as $S\to C$ is decomposable by \eqref{excinv-1}. This implies that $s_1'^2 = -\seg(S)$ and $s_2'^2 = \seg(S)$ (see Lemma~\ref{segrenegative}\eqref{segreneg-2}). On the other hand, we have $s_2^2 = s_2'^2$ and $s_1^2 = s_1'^2 - m$. Thus $m=s_1'^2 - s_1^2=-2\seg(S)> 0$. In particular, $\eta$ corresponds to a birational map $\eta_n$ as in Lemma~\ref{exceptional}\eqref{exc-3}, and $\kappa$ is an exceptional conic bundle over $C$.
\hfill\qedhere	\end{enumerate}
\end{proof}

Combining Lemma~\ref{exceptionalmaximal} and Proposition~\ref{exceptionalinvolution}, we get the main result of this section. 

\begin{proposition}\label{exceptionaltheorem}
	Let $C$ be a curve of positive genus. Let $\kappa \colon X\to C$ be an exceptional conic bundle with two $(-n)$-sections $s_1$ and $s_2$. The contraction of an irreducible component in each singular fibre gives a birational morphism $\eta\colon X\to S$, where $\pi\colon S= \PP(\mathcal{O}_C(D)\oplus \mathcal{O}_C)\to C$ is a decomposable ruled surface, for some $D\in \Pic(C)$. In particular, $\kappa = \pi\eta$. Denote by $s_1',s_2'$ the images of $s_1,s_2$ by $\eta$, and by $Z\subset s_1'$, $P\subset s_2'$ the sets of basepoints of~$\eta^{-1}$. The algebraic group $\Aut(X)$ is maximal if and only if $-2D$ is linearly equivalent to 
	$$
	\sum_{z\in \pi(Z)} z - \sum_{p\in \pi(P)} p,
	$$
	and in this case, $\Aut(X)$ fits into an exact sequence of algebraic groups
	$$
	0 \longrightarrow \mathbb{G}_m\rtimes \mathbb{Z}/2\mathbb{Z} \longrightarrow \Aut(X) \overset{\kappa_*}{\longrightarrow} H,
	$$
	where $H$ denotes the subgroup of $\Aut(C)$ which fixes the finite subset $\pi(Z\cup P)$. Else, $\Aut(X)$ is not maximal and can be embedded in an infinite increasing sequence of algebraic subgroups of\, $\operatorname{Bir}(C\times \PP^1)$.
\end{proposition}

\begin{proof}
	The structure morphism $\kappa$ induces a morphism of algebraic groups $\kappa_*\colon \Aut(X)\to \Aut(C)$, and an element in the image of $\kappa_*$ must preserves $\pi(Z\cup P)$, which is the set of points in $C$ having singular fibres. The rest of the statement follows from Lemma~\ref{exceptionalmaximal} and Proposition~\ref{exceptionalinvolution}.
\end{proof}

\begin{corollary}\label{exceptionalnotmax}
	Let $C$ be a curve of positive genus. Then there exist exceptional conic bundles $X\to C$ such that $\Aut(X)$ is not a maximal algebraic subgroup of\, $\Bir(C\times \PP^1)$.
\end{corollary}

\begin{proof}
	Let $X$ be an exceptional conic bundle over $C$ which is not the blowup of a decomposable ruled surface $\pi\colon \PP(\mathcal{O}_C(D)\oplus \mathcal{O}_C)\to C$ along $F=\{p_1,p_2,\ldots,p_{2\deg(D)}\}$ lying in two disjoint sections $s_1$ and $s_2$ such that $-2D$ is linearly equivalent to
	$$
	\sum_{p\in s_1\cap F} \pi(p) - \sum_{p\in s_2\cap F} \pi(p).
	$$
	By Proposition~\ref{exceptionaltheorem}, $\Aut(X)$ is not a maximal algebraic subgroup of\, $\Bir(C\times \PP^1)$. 
\end{proof}

In Proposition~\ref{exceptionaltheorem}, the morphism $\kappa_*\colon \Aut(X) \to H$ can be surjective: it is always the case either if  $C=\PP^1$ (see \cite[Lemma 4.3.3(1)]{Blanc}), or if $C$ is a curve of genus $g\geq 2$ with a trivial automorphism group. We give an example where this surjectivity fails. 

\begin{example}
	Let $C$ be an elliptic curve over $\mathbb{C}$ with neutral element $p_0$. Choose a $4$-torsion point $p_1\in C$ such that $p_2=2p_1\neq p_0$, and denote by $\Delta = \{p_0,p_1,p_2,p_3\}$ the subgroup generated by $p_1$. Define the ruled surface $\pi\colon S=\PP(\mathcal{O}_C(D)\oplus \mathcal{O}_C)\to C$, where $D = p_0+p_1$. The line subbundle $\mathcal{O}_C(D) \subset \mathcal{O}_C(D)\oplus \mathcal{O}_C$ corresponds to a section of self-intersection $-\deg(D)=-2$ (see \cite[Proposition 2.15]{fong}). By Lemma~\ref{segrenegative}\eqref{segreneg-1}, it follows that $\sigma$ is the unique section of $\pi$ with negative self-intersection, and therefore $\seg(S)=-2$. Let $s_1',s_2'$ be two disjoint sections with $s_1'^2 = -2$ and $s_2'^2=2$. Denote by $\eta\colon X\to S$ the blowup of $s'_2\cap \pi^{-1}(\Delta)$, and by $s_1,s_2$ the strict transforms of $s_1'$ and $s_2'$ by $\eta$. Then $\kappa = \pi \eta$ is a conic bundle. Moreover, $$(p_0+p_1+p_2+p_3)-(2D) = -p_0-p_1+p_2+p_3=-(p_1-p_0)+(p_2-p_0)+(p_3-p_0) \sim 0$$
	implies that $(-2D) \sim -(p_0+p_1+p_2+p_3)$, and it follows that $\Aut(X)$ is maximal with $\Aut_C(X) \simeq \mathbb{G}_m\rtimes \mathbb{Z}/2\mathbb{Z}$  (see Proposition~\ref{exceptionaltheorem}).
	
	Let $f\in \Aut(C)$ be the translation $x\mapsto x+p_1$ which preserves $\Delta$; \textit{i.e.}\ $f\in H$. Denote by $\widetilde{H}$ the subgroup of $\Aut(X)$ which fixes $s_1$ and $s_2$. Notice that $\eta$ is $\widetilde{H}$-equivariant and the following diagram is commutative:
	\[
	\begin{tikzcd} [sep = 2em,/tikz/column 1/.style={column sep=0.1em},/tikz/column 2/.style={column sep=0.1em}]
		\Aut(X)\arrow[rrdd,"\kappa_*" swap] & \supset & \widetilde{H} \arrow[dd,"\kappa_*"]\arrow[rd,"\eta_*"] \\
		& &  & \Aut(S)\arrow[ld,"\pi_*"]\\
		\Aut(C) & \supset & H\rlap{.}
	\end{tikzcd}
	\] 
	Assume that $f\in \Aut(C)$ lifts to an element of $\Aut(X)$; then it can also be lifted in $\widetilde{H}$ (if the lifting permutes $s_1$ and $s_2$, compose it with the non-trivial involution to get an element in $\widetilde{H})$, and \textit{a fortiori} can be lifted in $\Aut(S)$. This is not the case because $f^*(D)=p_0+p_3$ is not linearly equivalent to $D$ .
\end{example}

\subsection{$\boldsymbol{(\mathbb{Z}/2\mathbb{Z})^2}$-conic bundles}

The key result in this section is Proposition~\ref{key}. We will also need \cite[Lemmas 4.4.1, 4.4.3, 4.4.4]{Blanc}. Their proofs are left as exercises in the original article. For the sake of self-containedness, we re-prove them below (see Lemmas~\ref{involution},~\ref{involutionnormalizer},~\ref{determinantinvolution}).

\begin{lemma}[{\textit{cf.}}~\protect{\cite[Lemma 4.4.1]{Blanc}}]\label{involution}
	Let $C$ be a curve. Every element of order two in $\operatorname{PGL}(2,\kk(C))$ is conjugate to an element of the form 
	$\sigma_f = \begin{bsmallmatrix}
		0 & f \\ 1 & 0
	\end{bsmallmatrix}$, 
	where $f\in \kk(C)^*$. Moreover, $\sigma_f$ and $\sigma_g $ are conjugate if and only if $f/g$ is a square in $\kk(C)^*$.
\end{lemma}

\begin{proof}
	Let $\sigma \in \PGL(2,\kk(C))$ be an element of order two, and let $v\in \PP^1$ be such that $\sigma(v) \neq v$. Since $\sigma$ is of order two, there exists an $f\in \kk(C)^*$ such that the matrix of $\sigma$ with respect to the basis $\{v,\sigma(v)\}$ is 
	$\sigma_f = \begin{bsmallmatrix}
		0 & f \\ 1 & 0
	\end{bsmallmatrix}$. Let $f,g\in \kk(C)^*$. Assume $\sigma_f$ and $\sigma_g$ are conjugate. Let $\widetilde{\sigma_f},\widetilde{\sigma_g}$ be their respective representatives in $\operatorname{GL}(2,\kk(C))$ having $1$ as the lower left coefficient. Then there exist  $\lambda \in \kk(C)$, $P\in \operatorname{GL}(2,\kk(C))$ such that $P\widetilde{\sigma_f}P^{-1}=\lambda \widetilde{\sigma_g}$. Taking the determinant in the last equality gives $f/g = \lambda^2$. Conversely, assume that $f/g=\lambda^2$ for some $\lambda \in \kk(C)^*$. Then 
	\begin{equation*}\pushQED{\qed} 
	\begin{bmatrix}
		1 & 0 \\ 0 & \lambda^{-1}
	\end{bmatrix}
	\cdot \sigma_g \cdot 
	\begin{bmatrix}
		1 & 0 \\ 0 & \lambda
	\end{bmatrix} = 
	\sigma_f.\qedhere \popQED
	\end{equation*}
\renewcommand{\qed}{}        
\end{proof}

\begin{lemma}[{\textit{cf.}}~\protect{\cite[Lemma 4.4.4]{Blanc}}]\label{involutionnormalizer}
	Assume that $\operatorname{char}(\kk)\neq2$. Let $C$ be a curve. Let $\sigma=\begin{bsmallmatrix} 0& f \\ 1 & 0 \end{bsmallmatrix}\in \operatorname{PGL}(2,\kk(C))$, where $f\in \kk(C)^*$ is not a square. Let $N_\sigma$ be the normalizer of\, $\langle\sigma\rangle$ in $\operatorname{PGL}(2,\kk(C))$. Then 
	$$
	N_\sigma = \left\{\begin{bmatrix} a &bf \\ b&a\end{bmatrix}\mathrel{}\middle|\mathrel{} a,b\in \kk(C)\right\} \cup \left\{\begin{bmatrix} a & -bf \\ b&-a\end{bmatrix}\mathrel{}\middle|\mathrel{} a,b\in \kk(C)\right\} \simeq N^\circ_\sigma \rtimes \mathbb{Z}/2\mathbb{Z},
	$$
	where $N^\circ_\sigma=\left\{\begin{bsmallmatrix} a &bf \\ b&a\end{bsmallmatrix}\,\middle| \,a,b\in \kk(C)\right\} $
        is isomorphic to $\kk(C)[\sqrt{f}]^*/\kk(C)^*$ via the group homomorphism $\begin{bsmallmatrix} a &bf \\ b & a \end{bsmallmatrix} \mapsto [a+b\sqrt{f}]$, and $\mathbb{Z}/2\mathbb{Z}$ is generated by the diagonal involution. The action of $\mathbb{Z}/2\mathbb{Z}$ on $N_\sigma^\circ$ sends $\begin{bsmallmatrix} a &bf \\ b & a \end{bsmallmatrix}$ onto $\begin{bsmallmatrix} a &-bf \\ -b & a \end{bsmallmatrix}$. 
\end{lemma}

\begin{proof}
	Since $\sigma$ has order two, the normalizer of $\langle\sigma\rangle$ equals the centralizer of $\langle \sigma \rangle$. Then it is a straightforward computation in $\operatorname{PGL}(2,\kk(C))$ to check that matrices commuting with $\sigma$ are of the form $\begin{bsmallmatrix} a & bf \\ b & a\end{bsmallmatrix}$ or $\begin{bsmallmatrix} a & -bf \\ b & -a\end{bsmallmatrix}$ for some $a,b\in \kk(C)$, and $N_\sigma^\circ$ is a normal subgroup of $N_\sigma$. Since $N_\sigma^\circ \cap \left\{I_2,\begin{bsmallmatrix} 1 & 0 \\ 0 & -1\end{bsmallmatrix}\right\}=\{I_2\}$ and $N_\sigma^\circ \cdot \left\{I_2,\begin{bsmallmatrix} 1 & 0 \\ 0 & -1\end{bsmallmatrix}\right\}=N_\sigma$, it follows that $N_\sigma \simeq N^\circ_\sigma \rtimes \mathbb{Z}/2\mathbb{Z}$. 
	For all $a_1,a_2,b_1,b_2\in \kk(C)$,
	$$
	\begin{bmatrix}
		a_1 & b_1f \\ b_1 & a_1
	\end{bmatrix}
	\cdot
	\begin{bmatrix}
		a_2 & b_2f \\ b_2 & a_2
	\end{bmatrix} = 
	\begin{bmatrix}
		a_1a_2+b_1b_2f & (a_1b_2 + a_2b_1)f \\ a_2b_1 + a_1b_2 & a_1a_2+b_1b_2f
	\end{bmatrix}
	$$
	and $(a_1+b_1\sqrt{f})(a_2+b_2\sqrt{f})=(a_1a_2+b_1b_2f)+(a_1b_2+a_2b_1)\sqrt{f}$. Hence $N_\sigma^\circ$ is isomorphic to $\kk(C)[\sqrt{f}]^*/\kk(C)^*$ via $\begin{bsmallmatrix} a &bf \\ b & a \end{bsmallmatrix} \mapsto [a+b\sqrt{f}]$. Finally, 
	$$
	\begin{bmatrix}
		1 & 0 \\0 & -1
	\end{bmatrix}
	\cdot  \begin{bmatrix} a &bf \\ b&a\end{bmatrix} \cdot
	\begin{bmatrix}
		1 & 0 \\0 & -1
	\end{bmatrix} = 
	\begin{bmatrix} a &-bf \\ -b&a\end{bmatrix}; 
	$$
	\textit{i.e.}\ the action of $\mathbb{Z}/2\mathbb{Z}$ on $N_\sigma^\circ$ sends $\begin{bsmallmatrix} a &bf \\ b & a \end{bsmallmatrix}$ onto $\begin{bsmallmatrix} a &-bf \\ -b & a \end{bsmallmatrix}$.
\end{proof}

The following key proposition is an analogue of \cite[Proposition 5.2.2]{Blanc} for non-rational ruled surfaces. We prove it by copying, \textit{mutatis mutandis}, the proof of \cite[Proposition 5.2.2]{Blanc}. 

\begin{proposition}[{\textit{cf.}}~\protect{\cite[Proposition 5.2.2]{Blanc}}]\label{key}
	Assume that $\operatorname{char}(\kk)\neq2$. Let $C$ be a curve of positive genus, and let $G$ be a finite subgroup of\, $\operatorname{Bir}(C\times \PP^1)=\operatorname{PGL}(2,\kk(C))\rtimes \Aut(C)$. Denote by $G'\subset G$ and $H\subset \Aut(C)$ the kernel and the image of the action of $G$ on the base of the fibration. Then the following hold:
	\begin{enumerate}
		\item\label{key-1} If $G'=\{1\}$, then $G$ is conjugate to $H$ in $\operatorname{Bir}(C\times \PP^1)$.
		\item\label{key-2} If $G' \simeq \mathbb{Z}/2\mathbb{Z}$ is generated by an involution with a non-trivial determinant, then $G$ normalizes a group $V\subset \operatorname{PGL}(2,\kk(C))$ isomorphic to $ (\mathbb{Z}/2\mathbb{Z})^2$ and containing $G'$.
	\end{enumerate}	
\end{proposition}

\begin{proof}
	By Tsen's theorem, $\kk(C)$ is a $C_1$-field. Then $H^1(H,\operatorname{GL}(2,\kk(C)))$ and $H^2(H,\kk(C)^*)$ are trivial (see \cite[Propositions X.3, X.10 and X.11]{Serre}), and this implies that $H^1(H,\operatorname{PGL}(2,\kk(C)))$ is also trivial.
	\begin{enumerate}[wide]
		\item If $G'=\{1\}$, then $G$ is isomorphic to $H$, and there exists a section $s\colon H\to G$. Let $s_{c}$ be the homomorphism $H \to \operatorname{PGL}(2,\kk)\rtimes H$, $h\mapsto (1,h)$ and $f$ be the homomorphism $H\to \operatorname{Bir}(C\times \PP^1)$, $h\mapsto s(h)s_{c}(h)^{-1}$. Denote by $\pi\colon \operatorname{Bir}(C\times \PP^1)\to \Aut(C)$ the projection onto $\Aut(C)$. For all $h\in H$, $\pi f (h)= \pi(s(h)) \pi(s_{c}(h)^{-1}) = 1$; \textit{i.e.}\ there exists a homomorphism $\tilde{f}\colon H \to \operatorname{PGL}(2,\kk(C))$, such that $f(h)=(\tilde{f}(h),1)$. In particular, $s(h)=(\tilde{f}(h),h)$. For all $h_1,h_2\in H$, $(\tilde{f}(h_1h_2),1) = f(h_1h_2)=s(h_1)s(h_2)s_c(h_2)^{-1}s_c(h_1)^{-1}=(\tilde{f}(h_1)h_1\cdot \tilde{f}(h_2),1)$; \textit{i.e.}\ $\tilde{f}$ is a cocycle. Since $H^1(H,\operatorname{PGL}(2,\kk(C)))$ is trivial, $f$ is conjugate to the trivial cocycle; this implies that $s$ and $s_c$ are conjugate up to an element of $\operatorname{PGL}(2,\kk(C))$. Thus $G$ and $H$ are conjugate.
		\item Let $\sigma$ be the element of order two of $G'$. From Lemma~\ref{involution}, we can assume up to conjugation that $\sigma=\begin{bsmallmatrix} 0&f\\1&0\end{bsmallmatrix}$ for some $f\in \kk(C)^*$ which is not a square (by assumption, the determinant of $\sigma$ is not trivial). We denote by $N_\sigma$ the normalizer of $\sigma$. By Lemma~\ref{involutionnormalizer}, $N_\sigma \simeq N_\sigma^\circ \rtimes \mathbb{Z}/2\mathbb{Z}$ with $N_\sigma^\circ$ isomorphic to $\kk(C)[\sqrt{f}]^*/\kk(C)^*$, and $\mathbb{Z}/2\mathbb{Z}$ acts on $N_\sigma^\circ$ by sending $\begin{bsmallmatrix} a & bf \\ b & a\end{bsmallmatrix}$ onto $\begin{bsmallmatrix} a & -bf \\ -b & a\end{bsmallmatrix}$. 
		
		All elements of the form $\begin{bsmallmatrix} a & -bf \\ b & -a\end{bsmallmatrix}$ have order two in $\operatorname{PGL}(2,\kk(C))$: the goal is to find one of them which  with $\sigma$ generates a subgroup $V$ isomorphic to $(\mathbb{Z}/2\mathbb{Z})^2$ and normalized by $G$.
		
		Let $h\in H$. Then $\pi^{-1}(h) \cap G=\{(\gamma, h),(\sigma \gamma,h)\}$ for some $\gamma \in \operatorname{PGL}(2,\kk(C))$. The element $\sigma$ has order two, and $G'$ is normal in $G$; this implies that $(\sigma,1)$ is in the centre of $G$. In particular, $(\gamma,h)(\sigma,1)(h^{-1}\cdot \gamma^{-1},h^{-1})=(\sigma,1)$, and it follows that $\gamma (h\cdot \sigma) \gamma^{-1} = \sigma$. By Lemma~\ref{involution}, there exists a $\mu\in \kk(C)^*$ such that $\mu^2 = f/(h\cdot f)$. Let $\beta = \begin{bsmallmatrix} \mu & 0 \\ 0 & 1 \end{bsmallmatrix}$ and $\alpha=\gamma \beta^{-1}$. Then $\beta (h\cdot \sigma) \beta^{-1} = \sigma$ and $\alpha \sigma \alpha^{-1} = \sigma$. Therefore, $\alpha \in N_\sigma$, and replacing $\mu$ with $-\mu$ if needed, we can assume that $\alpha \in N_\sigma^\circ$. Under this last further condition, $\mu$ and $\alpha^2$ are uniquely determined by $h$ since $(\sigma \alpha)^2 = \alpha^2$. By associating $h$ to $\rho_h=\alpha^2$ and $\mu_h=\mu$, this yields the following well-defined maps:
		\begin{align*}
			\rho \colon H & \longrightarrow N^\circ_\sigma & \mu\colon H &\longrightarrow \kk(C)^* \\
			h& \longmapsto \rho_h, & h&\longmapsto \mu_h.
		\end{align*}
		
		We show that $\mu$ is a cocycle and $\rho$ is also a cocycle after conjugating by some element of $\operatorname{PGL}(2,\kk(C))$. Let $h_1,h_2 \in H$ and $h_3=h_1h_2$. For $i\in \{1,2,3\}$, choose as previously $(\alpha_i\beta_i,h_i)\in G$, where $\alpha_i\in N^\circ_\sigma$, $\beta_i = \begin{bsmallmatrix} \mu_i & 0 \\ 0 & 1 \end{bsmallmatrix}$, $\mu_i^2 = f/(h_i\cdot f)$. We can also choose $\alpha_3\beta_3$ such that $(\alpha_1\beta_1,h_1)(\alpha_2\beta_2,h_2)= (\alpha_3\beta_3,h_3)$, which implies that 
		\begin{equation}\label{eq:cocycle}
			\alpha_3=\alpha_1\beta_1(h_1\cdot (\alpha_2\beta_2))\beta_3^{-1}=\alpha_1 (\beta_1(h_1\cdot \alpha_2)\beta_1^{-1}) \beta_1(h_1\cdot \beta_2)\beta_3^{-1}. \tag{$\dagger$}
		\end{equation}
		Writing explicitly $\alpha_2 = \begin{bsmallmatrix}a & bf \\ b & a\end{bsmallmatrix}$, it follows that
                  \[\beta_1(h_1\cdot \alpha_2) \beta_1^{-1} = \begin{bmatrix} \mu_1(h_1\cdot a) &(h_1\cdot b)f \\h\cdot b & \mu_1(h_1\cdot a)\end{bmatrix}\in N_\sigma^\circ.\]
                  Then $\beta_1(h_1\cdot \beta_2)\beta_3^{-1}\in N_\sigma^\circ$, which is a diagonal matrix and thus equals the identity. This implies that $\mu_3 = \mu_1(h_1\cdot \mu_2)$; \textit{i.e.}\ $\mu$ is a cocycle. The group $H^1(H,\kk(C)^*)$ is trivial (see \cite[Propositions X.10 and X.11]{Serre}), so there exists a $\nu\in \kk(C)^*$ such that $\mu_h=\nu/(h\cdot \nu)$ for all $h\in H$. Then $f/\nu^2$ is $H$-invariant. Conjugating $G$ by $\begin{bsmallmatrix} 1 & 0 \\ 0 & \nu \end{bsmallmatrix}$, we can assume that $f$ is $H$-invariant, which is equivalent to $\mu_h =1$ for all $h\in H$.
		From the equation (\ref{eq:cocycle}), it follows that $\alpha_3= \alpha_1(h_1\cdot \alpha_2)$, which implies that $\rho$ is a cocycle.
		
		The $H$-equivariant exact sequence $1 \to \kk(C)^* \to \kk(C)[\sqrt{f}]^* \to N^\circ_\sigma \to 1$ with the equalities $H^1(H,\kk(C)[\sqrt{f}]^*)=\{1\}$ and $H^2(H,\kk(C)^*)=\{1\}$ (see \cite[Propositions X.10 and X.11]{Serre}) imply that $H^1(H,N_\sigma^\circ)=\{1\}$. 
		Therefore, $\rho$ is conjugate to the trivial cocycle; \textit{i.e.}\ there exists a $\tau\in N_\sigma^\circ$ such that $\rho_h = \tau\cdot {(h\cdot \tau)}^{-1}$ for all $h\in H$. The element $T=(\tau,-1)\in N_\sigma^\circ\rtimes \mathbb{Z}/2\mathbb{Z}$ has order two and is different from $\sigma$ (because $\sigma\in N_\sigma^\circ$). Moreover, every element of $G$ is of the form $((\alpha,1),h)$ with $\alpha\in N_\sigma^\circ$ and $h\in H$ such that $\alpha^2=\rho_h$, and $((\alpha,1),h)(T,1)(h^{-1}\cdot (\alpha^{-1},1),h^{-1})=((\alpha,1)(h\cdot \tau,-1)(\alpha^{-1},1),1)=((\alpha^2(h\cdot \tau),-1),1)=(T,1)$ in~$G$. The subgroup generated by $\sigma$ and $(T,1)$ is normalized by $G$ and is isomorphic to $(\mathbb{Z}/2\mathbb{Z})^2$.
\qedhere	\end{enumerate}
\end{proof}

Under the assumptions of Proposition~\ref{key}, the automorphism group of a conic bundle $X$ such that $\Aut_C(X)\simeq \mathbb{Z}/2\mathbb{Z}$ is not maximal. Below, we see that a conic bundle $X$ such that $\Aut_C(X)\simeq (\mathbb{Z}/2\mathbb{Z})^2$ is always a $(\mathbb{Z}/2\mathbb{Z})^2$-conic bundle and has a maximal automorphism group (Lemmas~\ref{Z/2Z} and~\ref{Z/2Z^2max}).

\begin{lemma}\label{involutionsections}
	Assume that $\operatorname{char}(\kk)\neq 2$. Let $C$ be a curve, and let $\kappa\colon X\to C$ be a conic bundle having at least one singular fibre. Suppose there exists a non-trivial involution $f\in \Aut_C(X)$ fixing pointwise two sections $s_1$ and $s_2$. Then in each singular fibre, the sections $s_1$ and $s_2$ pass through different irreducible components.
\end{lemma}

\begin{proof}
	Assume there is a singular fibre $\kappa^{-1}(p)$ where $s_1$ and $s_2$ pass through the same irreducible component. Since $f$ fixes pointwise $s_1$ and $s_2$, the contraction of the other irreducible component gives an $f$-equivariant birational morphism $\eta\colon \kappa^{-1}(U) \to U\times \PP^1$, where $U$ is an open neighbourhood of $p$. The $C$-automorphism $\eta f \eta^{-1}$ has order two, which implies that there exist $a,b,c\in \mathcal{O}_C(U)$ such that 
	\begin{equation}\label{eq:notidentity}
		\eta f \eta^{-1}:(x,[u:v])\mapsto (x,[au+bv:cu-av]). \tag{$\star$}
	\end{equation}
	On the other hand, $\eta f \eta^{-1}$ fixes the point contracted by $\eta$ and the sections $\eta(s_1)$ and $\eta(s_2)$. In particular, it fixes three distinct points in the fibre $p_1^{-1}(p)$, where $p_1\colon U\times \PP^1\to U$ denotes the first projection. Therefore, $\eta f \eta^{-1}_{|p_1^{-1}(p)}$ equals the identity. It follows from (\ref{eq:notidentity}) that $b(p)=c(p)=0$ and $a(p)=-a(p)\neq 0$, which gives a contradiction since $\operatorname{char}(\kk)\neq 2$.
\end{proof}

\begin{lemma}[{\textit{cf.}}~\protect{\cite[Lemma 4.4.3]{Blanc}}]\label{determinantinvolution}
	Let $C$ be a curve. Let $\sigma\in \operatorname{PGL}(2,\kk(C))$ be a non-trivial involution with $\det(\sigma)\in \kk(C)^*/(\kk(C)^*)^2$.
	\begin{enumerate}
		\item\label{detinv-1} If\, $\det(\sigma) = 1$, then $\sigma$ is diagonalizable and fixes pointwise two sections.
		\item\label{detinv-2} If\, $\det(\sigma)\neq 1$, then $\sigma$ is not diagonalizable and fixes pointwise an irreducible curve which is birational to a $2$-to-$1$ cover of\, $C$ ramified above an even positive number of points.
	\end{enumerate}
\end{lemma}

\begin{proof}
	\begin{enumerate}[wide]
		\item If $\det(\sigma)=1$, then $\sigma$ is conjugate to the matrix $\begin{bsmallmatrix} 0 & 1 \\ 1 & 0\end{bsmallmatrix}$ by Lemma~\ref{involution}, and 
		$$
		\begin{bmatrix} 1 &1 \\ 1 & -1 \end{bmatrix} \cdot \begin{bmatrix} 0 & 1 \\ 1 & 0\end{bmatrix} \cdot \begin{bmatrix} 1 &1 \\1 & -1\end{bmatrix} = \begin{bmatrix} - 1 & 0 \\ 0 & 1\end{bmatrix}.
		$$
		In particular, $\sigma$ is diagonalizable and fixes two sections. 
		\item If $\det(\sigma)\neq 1$, then $\sigma$ is conjugate to $\sigma_f=\begin{bsmallmatrix} 0 & f \\ 1 & 0 \end{bsmallmatrix}$, where $f\in \kk(C)^*$ is not a square by Lemma~\ref{involution}. Assume $\sigma_f$ is conjugate to a diagonal matrix $\begin{bsmallmatrix} \alpha & 0 \\ 0 & \beta \end{bsmallmatrix}$. By taking the determinant, there exists a $\lambda\in \kk(C)^*$ such that $\lambda^2 \alpha \beta = f$. Since $\sigma$ has order two, $\alpha^2 = \beta^2 $, which implies that $f^2 = \lambda^4 \alpha^4$. This contradicts the assumption that $f$ is not a square.
		
		The equation $[f(x)v:u] = [u:v]$ is equivalent to $u^2 - v^2 f(x)=0$, which defines an irreducible curve $Q$ in $C\times \PP^1$. The first projection of $C\times \PP^1$ restricted to $Q$ is a $2$-to-$1$ cover of $C$, and $Q$ is birational to a smooth curve having an even positive number of ramification points by Hurwitz's formula \cite[Corollary IV.2.4]{Hartshorne}.
\qedhere	\end{enumerate}
\end{proof}

\begin{corollary}\label{segrenegativedeterminanttrivial}
	Let $C$ be a curve. Let $\pi\colon S\to C$ be a ruled surface such that $\seg(S)<0$, let $p_1\colon C\times \PP^1 \to C$ be the first projection, and let $f\colon S\dashrightarrow C\times \PP^1$ be a birational map such that $p_1f=\pi$. Then $f\Aut_C(S)f^{-1} \subset \operatorname{PSL}(2,\kk(C))$.
\end{corollary}

\begin{proof}
	Let $\sigma\in \Aut_C(S)$. Then $f\sigma f^{-1}\in \operatorname{PGL}(2,\kk(C))$. Since $\seg(S)<0$, there exists a section $s\colon C \to S$ of negative self-intersection which is fixed by $\sigma$. Assume that $\det(\sigma)\neq 1$; then $\sigma$ also fixes pointwise an irreducible curve which is a $2$-to-$1$ cover of $C$ by Lemma~\ref{determinantinvolution}\eqref{detinv-2}. Then $\sigma$ fixes three distinct points in a general fibre; hence $\sigma$ equals the identity, which contradicts the inequality $\det(\sigma)\neq 1$.
\end{proof}

\begin{lemma}\label{Z/2Z}
	Let $C$ be a curve of positive genus. Let $\kappa\colon X\to C$ be a ruled surface, or a conic bundle with at least one singular fibre such that its two irreducible components are exchanged by an element of $\Aut(X)$. If $\Aut_C(X)\simeq (\mathbb{Z}/2\mathbb{Z})^r$ with $r\in \{1,2\}$, then $\det(\sigma)\neq 1$ for all $\sigma\in \Aut_C(X)\setminus \{1\}$. In particular, if $r=2$, then $\kappa$ is a $(\mathbb{Z}/2\mathbb{Z})^2$-conic bundle. 
\end{lemma}

\begin{proof}
	First assume that $\kappa$ has at least one singular fibre, and let $\eta \colon X\to S$ be the contraction of an irreducible component in each singular fibre. Let $G$ be the normal subgroup of $\Aut_C(X)$ which leaves invariant each irreducible component of the singular fibres.
	
	Suppose that $r=1$. Let $\sigma_1\in \Aut_C(X)$ be the element of order two, and assume that $\det(\sigma_1)=1$. By Lemma~\ref{determinantinvolution}\eqref{detinv-1}, the automorphism $\sigma_1$ fixes pointwise two sections which do not pass through the same irreducible components in each singular fibre (see Lemma~\ref{involutionsections}), and $\eta$ is $\sigma_1$-equivariant. This implies that $\eta G \eta^{-1}$ has an invariant section, and by Lemma~\ref{invariantexceptional}\eqref{invariantexc-1}, $\kappa$ is an exceptional conic bundle, which contradicts the assumption that $\Aut_C(X)$ is finite (see Lemma~\ref{exceptionalmaximal}).
	
	Suppose that $r=2$. Let $\sigma_1,\sigma_2,\sigma_3$ be the elements of order two in $\Aut_C(X)$. If $\det(\sigma_i) \neq 1$ for all $i\in\{1,2,3\}$, then $\kappa$ is a $(\mathbb{Z}/2\mathbb{Z})^2$-conic bundle by Lemma~\ref{determinantinvolution}\eqref{detinv-2}. Without loss of generality, we can assume by contradiction that $\det(\sigma_1)=1$. By Lemma~\ref{determinantinvolution}\eqref{detinv-1}, the automorphism $\sigma_1$ fixes two sections $s_1$ and $s_2$ which do not pass through the same irreducible components in each singular fibre (see Lemma~\ref{involutionsections}), and $\eta$ is $\sigma_1$-equivariant.
	In particular, $\sigma_1\in G$, and $G$ is not trivial. Let $\sigma_j\in G$, and denote by $\Fix(\sigma_1)$ the set of fixed points of $\sigma_1$. Then $\sigma_j(\Fix(\sigma_1)) = \Fix(\sigma_1)$. It follows that $\sigma_j$ either permutes $s_1$ and $s_2$ or leaves them invariant. Since we assume that $\sigma_j\in G$, it follows that $\sigma_j$ must leave them invariant. This implies that $G$ leaves $s_1$ and $s_2$ invariant, and by Lemma~\ref{invariantexceptional}\eqref{invariantexc-1}, $\kappa$ is an exceptional conic bundle. This contradicts the assumption that $\Aut_C(X)$ is finite (see Lemma~\ref{exceptionalmaximal}).
	
	Assume from now on that $\kappa$ has no singular fibre and $\Aut_C(X)\simeq(\mathbb{Z}/2\mathbb{Z})^r$ for $r\in\{1,2\}$. Then $X$ is a ruled surface with $\seg(X)>0$; see  \cite[Theorem 2]{Maruyama}. Let $\sigma_1$ be an element of order two with $\det(\sigma_1)=1$; then $\sigma_1$ fixes two sections $s_1$ and $s_2$ (see Lemma~\ref{determinantinvolution}\eqref{detinv-1}) which intersect because $X$ is indecomposable; see \cite[Proposition 2.18(1)]{fong}. Applying elementary transformations centred on the intersections yields a $\sigma_1$-equivariant birational map $\phi\colon X\dashrightarrow S$, where $\pi\colon S\to C$ is a decomposable ruled surface. The automorphism $\phi \sigma_1 \phi^{-1}$ fixes the strict transforms of $s_1$ and $s_2$ which are disjoint sections of $\pi$ and the basepoints of $\phi^{-1}$ which are in the complement of $s_1\cup s_2$. Choose trivializations $(U_i)_i$ of $\pi$ such that the strict transforms of $s_1$ and $s_2$ are the zero and infinity sections of $\pi$. Then we have $(\phi \sigma_1 \phi^{-1})_{|U_i}\colon (x,[u:v])\mapsto (x,[\alpha_i(x) u:v])$ for some $\alpha_i\in \mathcal{O}_C(U_i)^*$, and the transition maps are of the form $s_{ij}\colon U_j\times \PP^1  \dashrightarrow U_i\times \PP^1$, $(x,[u:v])\mapsto (x,[t_{ij}(x)u:v])$ for some $t_{ij}\in \mathcal{O}_C(U_{ij})^*$. The condition $ (\phi \sigma_1 \phi^{-1})_{|U_i} s_{ij} = s_{ij} (\phi \sigma_1 \phi^{-1})_{|U_j}$ implies that there exists an $\alpha \in \mathbb{G}_m$ such that $\alpha_i=\alpha_j = \alpha$ for all $i,j$. Since $\phi \sigma_1 \phi^{-1}$ fixes the basepoints of $\phi^{-1}$ which do not lie in $s_1 \cup s_2$, this implies that $\alpha=1$; \textit{i.e.}\ $\sigma_1$ equals identity, which gives a contradiction. 
\end{proof}

\begin{lemma}\label{Z/2Z^2max}
	Assume that $\operatorname{char}(\kk)\neq 2$. Let $C$ be a curve of positive genus, and let $\kappa\colon X\to C$ be a $(\mathbb{Z}/2\mathbb{Z})^2$-conic bundle. Then $\Aut(X)$ is a maximal algebraic subgroup of\, $\operatorname{Bir}(C\times \PP^1)$. Moreover, $\Aut(X)$ fits into an exact sequence
	$$
	1 \longrightarrow (\mathbb{Z}/2\mathbb{Z})^2 \longrightarrow \Aut(X) \overset{\pi_*}{\longrightarrow} \Aut(C),
	$$
	where the image of $\pi_*$ equals $\Aut(C)$ if $X\simeq \mathcal{A}_1$ $($or, equivalently, $C$ is an elliptic curve and $\kappa$ is the only $(\mathbb{Z}/2\mathbb{Z})^2$-ruled surface over $C)$ and otherwise equals a finite subgroup of $\Aut(C)$ preserving the set of singular fibres of $\kappa$ $($which is possibly empty if $\kappa$ is a ruled surface over a curve $C$ of genus $\geq 2)$.
\end{lemma}

\begin{proof}
	Let $p\in C$ be such that $\kappa^{-1}(p)$ is a smooth fibre. The group $\Aut_C(X)\simeq (\mathbb{Z}/2\mathbb{Z})^2$ acts on $\kappa^{-1}(p)$. Any $\sigma \in \Aut_C(X)$ of order two is the form 
	$\begin{bsmallmatrix}
		a & b \\ c & -a
	\end{bsmallmatrix}$ 
	for some $a,b,c\in \mathcal{O}_C(U)$, where $U$ is an open neighbourhood of $p$. If $\sigma$ equals identity on $\kappa^{-1}(p)$, then $b(p)=c(p)=0$ and $a(p)=-a(p)\neq 0$. Hence $\sigma$ does not act trivially on $\kappa^{-1}(p)$; \textit{i.e.}\ the action of $(\mathbb{Z}/2\mathbb{Z})^2$ over $\kappa^{-1}(p)$ is faithful. Moreover, any $(\mathbb{Z}/2\mathbb{Z})^2 \subset \operatorname{PGL}(2,\kk)$ acts without fixed points,
        and by Lemma~\ref{conicequivariant}, the subgroup $\Aut(X)$ is maximal.
	
	It remains to prove the exact sequence. By definition, the kernel of $\pi_*$ is isomorphic to $(\mathbb{Z}/2\mathbb{Z})^2$. First assume that $\kappa$ is a ruled surface. If $g=1$, the assumption that $\Aut_C(S)$ is finite implies that $\seg(X)>0$ (see \cite[Theorem 2]{Maruyama}) and $\kappa$ is an indecomposable ruled surface (see \cite[Proposition 2.18(1)]{fong}). Then $X$ is $C$-isomorphic to $\mathcal{A}_0$ or $\mathcal{A}_1$ (see \cite[Theorem V.2.15]{Hartshorne}), which respectively satisfy $\seg(\mathcal{A}_0)=0$ and $\seg(\mathcal{A}_1)=1$ (see \cite[Proposition 2.21]{fong}).
	Therefore, $X$ is isomorphic to $\mathcal{A}_1$, and the exact sequence follows from \cite[Theorem 3(4)]{Maruyama}. If $g\geq 2$, the statement holds because $\Aut(C)$ is a finite group (see \cite[Exercise IV.2.5]{Hartshorne}). Assume that $\kappa$ has a singular fibre. Then any element of $\Aut(X)$ has to preserve to set of singular fibres, which is finite. It follows that the morphism $\kappa_*\colon \Aut(X) \to \Aut(C)$ has  finite image.
\end{proof}

\subsection{Ruled surfaces}

\begin{proposition}\label{autoruled}
	Assume that $\operatorname{char}(\kk)\neq2$. Let $C$ be a curve of genus $g\geq 1$ and $\pi\colon S\to C$ be a ruled surface. The following hold:
	\begin{enumerate}
		\item If\, $\pi$ is trivial, then $\Aut(S)$ is maximal.
		\item If\, $\seg(S)=0$, $\pi$ is not trivial and $S\simeq \PP(\mathcal{O}_C(D)\oplus \mathcal{O}_C)$ is decomposable, then $\Aut(S)$ is maximal if and only if $g=1$, or $g\geq 2$ and $2D$ is principal. If $g\geq 2$ and $2D$ is not principal, then $\Aut(S)$ can be embedded in an infinite increasing sequence of algebraic subgroups of\, $\operatorname{Bir}(C\times \PP^1)$.
		\item\label{autoruled-3} If\, $\seg(S)=0$ and $S$ is indecomposable, then $\Aut(S)$ is maximal if and only if $g=1$. If $g\geq 2$, then $\Aut(S)$ can be embedded in an infinite increasing sequence of algebraic subgroups of\, $\operatorname{Bir}(C\times \PP^1)$.
		\item If\, $\seg(S)>0$, then $\Aut(S)$ is maximal if and only if $S$ is a $(\mathbb{Z}/2\mathbb{Z})^2$-ruled surface.
	\end{enumerate}
\end{proposition}

\begin{proof}
	\begin{enumerate}[wide]
		\item If $S$ is trivial, each $\Aut(S)$-orbit contains at least a fibre. Hence $\Aut(S)$ is maximal by Lemma~\ref{conicequivariant}.
		\item If $g=1$, then the morphism $\Aut^\circ(S)\to \Aut^\circ(C)$ is surjective; see \cite[Proposition 3.9]{fong}. There is no $\Aut(S)$-orbit of finite dimension; hence $\Aut(S)$ is maximal. 
		
		Assume that $g\geq 2$. If $2D$ is principal, then $\Aut_C(S) \simeq \mathbb{G}_m\rtimes \mathbb{Z}/2\mathbb{Z}$ (see Lemma~\ref{decomposableseg0}\eqref{decompseg0-2}), and the group $\Aut_C(S)$ acts on a fibre with two orbits: one is isomorphic to $\mathbb{G}_m$, and the other one is made of two points exchanged by the involution. From Lemma~\ref{conicequivariant}, $\Aut(S)$ is a maximal algebraic subgroup. If $2D$ is not principal, $\Aut(S)\simeq \mathbb{G}_m$ (see Lemma~\ref{decomposableseg0}\eqref{decompseg0-3}). Take a point in a minimal section; its $\Aut(S)$-orbit is finite and contains at most one point in each fibre. The blowup of this orbit followed by the contraction of the strict transforms of the fibres gives an $\Aut(S)$-equivariant birational map $S\dashrightarrow S'$, where $S'$ is a ruled surface with $\seg(S')<0$. Then apply Lemma~\ref{infiniteseq}.
		\item From \cite[Proposition 2.18(3.ii)]{fong}, $S$ has a unique minimal section which is $\Aut(S)$-invariant. If $g \geq 2$, take a point on this minimal section and blow up its orbit (which consists of finitely many points on the minimal section), then contract the strict transforms of the fibres. This defines an $\Aut(S)$-equivariant map $S\dashrightarrow S'$ with $\seg(S')<0$. Then apply Lemma~\ref{infiniteseq}. If $g=1$, then $S\simeq \mathcal{A}_0$ and $\pi_*\colon \Autzero(S) \to \Autzero(C)$ is surjective; see \cite[Proposition 3.6]{fong}. In particular, there is no $\Aut(S)$-orbit of dimension $0$ and no $\Aut(S)$-equivariant map. Thus $\Aut(S)$ is maximal.
		\item Assume that $\seg(S)>0$. In particular, $\pi$ is indecomposable (see \textit{e.g.}~\cite[Proposition 2.18(1)]{fong}). If $g=1$, then $S\simeq \mathcal{A}_1$ is a $(\mathbb{Z}/2\mathbb{Z})^2$-ruled surface and $\Aut(\mathcal{A}_1)$ is maximal by Lemma~\ref{Z/2Z^2max}. From now on, we assume that $g\geq 2$. By \cite[Lemma 3, Theorem 2]{Maruyama}, $\Aut_C(S)$ is isomorphic to a subgroup of $\Pic^0(C)[2]$. In particular, it is a finite subgroup of $\operatorname{PGL}(2,\kk)$ such that every element is an involution. Hence $\Aut_C(S) \simeq (\mathbb{Z}/2\mathbb{Z})^s$ for some $s\in \{0,1,2\}$. Moreover, $\Aut(C)$ is finite, and this implies that $\Aut(S)$ is also finite.
		
		  By Lemma~\ref{Z/2Z}, each non-trivial element of $\Aut_C(S)$ has a non-trivial determinant. If $s=0$, then $\Aut(S)$ is conjugate to a finite subgroup of $\Aut(C)\subsetneq \Aut(C\times \PP^1)$ by Proposition~\ref{key}\eqref{key-1}.  If $s=1$, then by Proposition~\ref{key}\eqref{key-2}, $\Aut(S)$ normalizes a group $V\simeq (\mathbb{Z}/2\mathbb{Z})^2$ containing $\Aut_C(S)$; \textit{i.e.}\ there exists a finite subgroup $G\subset \Bir(C\times \PP^1)$ containing $\Aut(S)$ such that $V$ is the kernel of the action of $G$ on $C$. In particular, $\Aut(S)\subsetneq G$. Therefore, we get that $\Aut(S)$ is not maximal if $s\in \{0,1\}$. If $s=2$, then $S$ is a $(\mathbb{Z}/2\mathbb{Z})^2$-ruled surface. Conversely, the automorphism group of a $(\mathbb{Z}/2\mathbb{Z})^2$-ruled surface is maximal by Lemma~\ref{Z/2Z^2max}.
\qedhere	\end{enumerate}
\end{proof}

\subsection{Examples of $\boldsymbol{(\mathbb{Z}/2\mathbb{Z})^2}$-conic bundles}

If $C$ is an elliptic curve, the Atiyah bundle $\mathcal{A}_1$ is the only $ (\mathbb{Z}/2\mathbb{Z})^2$-ruled surface. We give below examples of $ (\mathbb{Z}/2\mathbb{Z})^2$-conic bundles over any curve $C$ of genus $g\geq 2$. If $X$ is a $(\mathbb{Z}/2\mathbb{Z})^2$-conic bundle over $\PP^1$ with at least one singular fibre, then every element of order two in $\Aut_{\PP^1}(X)$ acts non-trivially on $\operatorname{Pic}(X)$ by permuting the irreducible components of a singular fibre (by \cite[Lemma 4.3.5]{Blanc}. If there exists an element of $\Aut_{\PP^1}(X)\setminus \{1\}$ acting trivially on $\Pic(X)$, then $X\to \PP^1$ is an exceptional conic bundle, and thus not a $(\mathbb{Z}/2\mathbb{Z})^2$-conic bundle by \cite[Lemmas 4.3.3(1) and 4.3.5]{Blanc}). The following example also shows that this does not hold anymore when $C$ has positive genus.

\begin{example}
	Assume that $\operatorname{char}(\kk)\neq 2$. Let $C$ be a curve of genus $g\geq 1$ and $D$ be a non-principal divisor such that $2D$ is principal. Let $S$ be the decomposable ruled surface $\PP(\mathcal{O}_C(D)\oplus \mathcal{O}_C)$. From Lemma~\ref{decomposableseg0}, $\Aut_C(S) = \mathbb{G}_m\rtimes \mathbb{Z}/2\mathbb{Z}$, and the element $\sigma$ of order two that generates $\mathbb{Z}/2\mathbb{Z}$ is conjugate to $\begin{bsmallmatrix} 0 & f \\ 1 & 0 \end{bsmallmatrix}$ for some $f\in \kk(C)^*$ such that $\operatorname{div}(f)=2D$. Since $D$ is not principal, $f$ is not a square. In particular, $\det(\sigma)\neq 1$ and $\sigma$ fixes pointwise an irreducible curve birational to a $2$-to-$1$ cover of $C$ and ramified above an even positive number of points (see Lemma~\ref{determinantinvolution}\eqref{detinv-2}). 
	
	The matrix $\tau=\begin{bsmallmatrix}
		a & -bf \\ b & -a
	\end{bsmallmatrix}$ of order two has  determinant $-a^2 +b^2f =- N(a+b\sqrt{f})$, where $N\colon\kk(C)[\sqrt{f}]\to\kk(C) $ is the norm, which is surjective (see \cite[Propositions X.10 and X.11]{Serre}). Choose $a$ and $b$ such that $\det(\tau)$ has a pole with odd multiplicity at a point where $f$ is regular. Then $\det(\tau)\neq 1$ and $\det(\sigma\tau)\neq 1$.
	
	Since $\sigma \tau = \tau \sigma$, the subgroup of $\operatorname{PGL}(2,\kk(C))$ generated by $\sigma$ and $\tau$ is isomorphic to $(\mathbb{Z}/2\mathbb{Z})^2$. Apply Propositions~\ref{Iskovskikh} and~\ref{reduction}; there exists a $\mathbb{Z}/2\mathbb{Z}$-equivariant birational map from $S$ to a conic bundle $X$ such that $(\mathbb{Z}/2\mathbb{Z})^2\subset \Aut_C(X)$, and $X$ is a ruled surface, or an exceptional conic bundle, or a $(\mathbb{Z}/2\mathbb{Z})^2$-conic bundle. Assume that $X$ is not a $(\mathbb{Z}/2\mathbb{Z})^2$-conic bundle. Notice that $X$ also cannot be a ruled surface with $\seg(X)<0$ by Corollary~\ref{segrenegativedeterminanttrivial} or $\seg(X)>0$ by \cite[Lemma 3]{Maruyama}. This implies $X$ is either a ruled surface with $\seg(X)=0$ or an exceptional conic bundle such that $\Aut_C(X)\simeq \mathbb{G}_m\rtimes \mathbb{Z}/2\mathbb{Z}$ (see Lemmas~\ref{decomposableseg0} and~\ref{exceptionalmaximal}) with $\det((-1,0))=1$. Since $\det(\sigma)$, $\det(\tau)$ and $\det(\sigma \tau )$ are all non-trivial; this gives a contradiction. Therefore, $X$ is a $(\mathbb{Z}/2\mathbb{Z})^2$-conic bundle.

\end{example}

\section{Proofs of the results}

\begin{proof}[Proof of Theorem~\ref{A}]
	Each algebraic group in the list is a maximal algebraic subgroup of $\operatorname{Bir}(C\times \PP^1)$ by Lemma~\ref{Z/2Z^2max} and Propositions~\ref{exceptionaltheorem} and~\ref{autoruled}. Conversely, let $G$ be a maximal algebraic subgroup of $\operatorname{Bir}(C\times \PP^1)$, where $C$ is a curve of genus $g\geq 1$. Using the regularization theorem (Proposition~\ref{regularization}) and the $G$-equivariant MMP (Proposition~\ref{Iskovskikh}), it follows that $G$ is conjugate to $\Aut(X)$ for some conic bundle $\kappa\colon X\to C$. If $\kappa$ has no singular fibre, directly apply Proposition~\ref{autoruled}. Else $\kappa$ has at least one singular fibre. If there is no element of $\Aut(X)$ permuting two irreducible components of a singular fibre, then there exists an $\Aut(X)$-equivariant contraction $X\to S$, where $S$ is a ruled surface; apply Proposition~\ref{autoruled} to conclude. Else, apply Proposition~\ref{reduction} with Proposition~\ref{key}; it follows that either $X$ is an exceptional conic bundle, or $\Aut(X)$ is conjugate to a subgroup of $\Aut(C\times \PP^1)$ or $\Aut(X')$, where $X'$ is a $(\mathbb{Z}/2\mathbb{Z})^2$-conic bundle. To conclude, apply Proposition~\ref{exceptionaltheorem} for the case of exceptional conic bundles, and apply Lemma~\ref{Z/2Z^2max} for the case of $(\mathbb{Z}/2\mathbb{Z})^2$-conic bundles. Finally, the exact sequences of \eqref{A-4} in the case $g=1$ and \eqref{A-5} are taken from \cite[Theorem 3]{Maruyama}.
\end{proof}

\begin{proof}[Proof of Corollary~\ref{D}]
	From \cite[Theorem 1]{Blanc}, every algebraic subgroup of $\operatorname{Bir}(\PP^2)$ is included in a maximal one. From Theorem~\ref{A}, every maximal algebraic subgroups of $\operatorname{Bir}(C\times \PP^1)$ has dimension at most four. By Remark~\ref{dimensionarbitrarylarge}, there exist algebraic subgroups of arbitrary large dimension, and they cannot be subgroups of the maximal ones. 
\end{proof}


\end{document}